\documentclass[prodmode,acmtoms]{acmsmall} 

\if{
\usepackage[ruled]{algorithm2e}

\SetAlFnt{\small}
\SetAlCapFnt{\small}
\SetAlCapNameFnt{\small}
\SetAlCapHSkip{0pt}
\IncMargin{-\parindent}
}\fi
\if{
\acmVolume{9}
\acmNumber{4}
\acmArticle{39}
\acmYear{2010}
\acmMonth{3}
}\fi

\usepackage{listings}

\newcommand{\code}[1]{\lstinline{#1}}

\usepackage[utf8]{inputenc}

\usepackage{lmodern}
\usepackage{graphicx}  
\usepackage[T1]{fontenc}
\usepackage{amsmath} 
\usepackage{amssymb}  
\usepackage{color}
\usepackage{enumitem}
\usepackage{myalgo}
\usepackage{hyperref}
\usepackage{subfigure}
\usepackage{enumerate}
\usepackage{multirow}
\usepackage{tikz, pgfplots}

\newcommand{\setD}{\mathcal{D}} 
\newcommand{\suppf}[1]{\text{supp}(#1)}

\newcommand{\R}{\mathbb{R}}
\newcommand{\F}{\mathbb{F}}

\newcommand{\N}{\mathbb{N}}
\newcommand{\x}{\mathbf{x}}
\newcommand{\e}{\mathbf{e}}
\newcommand{\y}{\mathbf{y}}

\newcommand{\nbenchs}{30}
\newcommand{\alphab}{\boldsymbol{\alpha}}

\newcommand{\deltab}{\boldsymbol{\delta}}
\usepackage{textgreek}
\renewcommand{\b}{\mathbf{b}}
\newcommand{\f}{\mathbf{f}}

\def\P{\mathbf{P}}
\def\Q{\mathbf{Q}}
\def\L{\mathbf{L}}
\def\D{\mathbf{D}}

\def\m{\mathbf{m}}

\def\f{f}
\def\a{\mathbf{a}}
\def\m{\mathbf{m}}

\def\S{\mathbf{S}}

\def\E{\mathbf{E}}
\def\K{\mathbf{K}}
\def\S{\mathbf{S}}
\def\Q{\mathbf{Q}}
\def\X{\mathbf{X}}

\renewcommand{\prec}{\text{prec}}

\newcommand{\Kpol}{\K_{\text{poly}}}

\newcommand{\iaboundfun}[2]{\mathtt{ia\_bound}(#1, #2)}
\newcommand{\iabound}{\mathtt{ia\_bound}}
\newcommand{\sdpboundfun}[3]{\mathtt{sdp\_bound}(#1, #2, #3)}
\newcommand{\sdpbound}{\mathtt{sdp\_bound}}
\newcommand{\boundfun}[7]{\mathtt{bound}(#1, #2, #3, #4, #5, #6, #7)}
\newcommand{\bound}{\mathtt{bound}}
\newcommand{\boundnlprogfun}[7]{\mathtt{bound\_nlprog}(#1, #2, #3, #4, #5, #6, #7)}
\newcommand{\boundnlprog}{\mathtt{bound\_nlprog}}
\newcommand{\sdppolyfun}[3]{\mathtt{sdp\_poly}(#1, #2, #3)}
\newcommand{\sdppoly}{\mathtt{sdp\_poly}}
\newcommand{\liftfun}[3]{\mathtt{lift}(#1, #2, #3)}
\newcommand{\lift}{\mathtt{lift}}
\newcommand{\poly}{_\text{poly}}
\newcommand{\sa}{_\text{sa}}

\newcommand{\sdpsa}{\mathtt{sdp\_sa}}

\newcommand{\sdptransc}{\mathtt{sdp\_transc}}
\newcommand{\sthreefp}{\mathtt{s3fp}}
\newcommand{\realtofloat}{\mathtt{Real2Float}}
\newcommand{\smtcoq}{\mathtt{smtcoq}}

\newcommand{\hol}{\text{\sc Hol-light}}

\newcommand{\bop}{\mathtt{bop}}
\newcommand{\coq}{\text{\sc Coq}}
\newcommand{\ocaml}{\text{\sc OCaml}}
\newcommand{\rosa}{\mathtt{Rosa}}
\newcommand{\sdpa}{\text{\sc Sdpa}}
\newcommand{\fptaylor}{\mathtt{FPTaylor}}
\newcommand{\nlcertify}{\mathtt{NLCertify}}
\makeatletter
\newcommand*{\circled}{\@ifstar\circledstar\circlednostar}
\newcommand*{\squared}{\@ifstar\squaredstar\squarednostar}
\makeatother

\newcommand*\circledstar[1]{%
  \tikz[baseline=(C.base)]
    \node[%
      fill,
      circle,
      minimum size=1.em,
      text=white,
      inner sep=0.5pt
    ](C) {\texttt{#1}};%
}
\newcommand*\circlednostar[1]{%
  \tikz[baseline=(C.base)]
    \node[%
      draw,
      circle,
      minimum size=1.em,
      inner sep=0.5pt
    ](C) {\texttt{#1}};%
}
\newcommand*\squaredstar[1]{%
  \tikz[baseline=(C.base)]
    \node[%
      fill,
      rectangle,
      minimum size=1.em,
      text=white,
      inner sep=0.5pt
    ](C) {\texttt{#1}};%
}
\newcommand*\squarednostar[1]{%
  \tikz[baseline=(C.base)]
    \node[%
      draw,
      rectangle,
      minimum size=1.em,
      inner sep=0.5pt
    ](C) {\texttt{#1}};%
}

\if{
\theoremstyle{plain}
\newtheorem{theorem}{Theorem}[section]
\newtheorem{corollary}[theorem]{Corollary}
\newtheorem{lemma}[theorem]{Lemma}
\newtheorem{proposition}[theorem]{Proposition}
\newtheorem{example}{Example}

\theoremstyle{definition}
\newtheorem{definition}{Definition}

\theoremstyle{remark}
\newtheorem{remark}{Remark}
}\fi

\if{
\newtheorem{theorem}{Theorem}[section]
\newtheorem{lemma}[theorem]{Lemma}

\theoremstyle{plain}
\newtheorem{definition}[theorem]{Definition}

\newtheorem{example}{Example}

\newtheorem{remark}{Remark}

}\fi

\begin{document}






\title{Certified Roundoff Error Bounds Using Semidefinite Programming}


\author{VICTOR MAGRON
\affil{CNRS Verimag}
GEORGE CONSTANTINIDES 
\affil{Imperial College London}
ALASTAIR DONALDSON
\affil{Imperial College London}
}




\begin{abstract}
Roundoff errors cannot be avoided when implementing numerical programs with finite precision.         
The ability to reason about rounding is especially important if one wants to explore a range of potential representations, for instance for FPGAs or custom hardware implementations. This problem becomes challenging when the program does not employ solely linear operations, and non-linearities are inherent to many interesting computational problems in real-world applications. 

Existing solutions to reasoning possibly lead to either inaccurate bounds or high analysis time in the presence of nonlinear correlations between variables.
Furthermore, while it is easy to implement a straightforward method such as interval arithmetic, sophisticated techniques are less straightforward to implement in a formal setting. Thus there is a need for methods which output certificates that can be formally  validated inside a proof assistant.

We present a framework to provide upper bounds on absolute roundoff errors of floating-point nonlinear programs. 
This framework is based on optimization techniques employing semidefinite programming and sums of squares certificates, which can be checked inside the Coq theorem prover to provide formal roundoff error bounds for polynomial programs.
Our tool covers a wide range of nonlinear programs, including polynomials and transcendental operations as well as conditional statements.                                                         
We illustrate the efficiency and  precision of this tool on non-trivial programs coming from biology, optimization and space control. 
Our tool produces more accurate error bounds for $23$ \% of all programs and yields better performance in $66$ \% of all programs.
\end{abstract}

\if{
}\fi




\if{
}\fi

\acmformat{Victor Magron, George Constantinides and Alastair Donaldson, 2016. Certified Roundoff Error Bounds Using Semidefinite Programming.}

\begin{bottomstuff}
This work was partly funded by the Engineering and Physical Sciences Research Council (EPSRC) ``Challenging Engineering'' Grant (EP/I020457/1, EP/K034448/1, EP/K015168/1 ), Royal Academy of Engineering, Imagination Technologies and European Research Council (ERC) ``STATOR'' Grant Agreement nr. 306595.

Author's addresses: V. Magron, CNRS Verimag, 700 avenue Centrale, 38401 Saint-Martin d'H\`eres FRANCE; 
G. Constantinides, Imperial College London, London SW7 2AZ, UK;
A. Donaldson, Imperial College London, London SW7 2AZ, UK.
\end{bottomstuff}


\maketitle
\section{INTRODUCTION} 
\label{sec:intro}
Constructing numerical programs which perform accurate computation turns out to be difficult, due to finite numerical precision of implementations such as floating-point or fixed-point representations. Finite-precision numbers induce roundoff errors,  and knowledge of the range of these roundoff errors is required to fulfill safety criteria of critical programs, as typically arising in modern embedded systems such as aircraft controllers. Such a knowledge can be used in general for developing accurate numerical software, but is also particularly relevant when considering migration of algorithms onto hardware (e.g. FPGAs). The advantage of architectures based on FPGAs is that they allow more flexible choices in number representations, rather than limiting the choice between  IEEE standard single or double precision. Indeed, in this case, we benefit from a more flexible number representation while still ensuring guaranteed bounds on the program output. 

To obtain lower bounds on roundoff errors, one can rely on testing approaches, such as meta-heuristic search~\cite{Borges12Test} or under-approximation tools (e.g.~$\sthreefp$~\cite{Chiang14s3fp}). Here, we are interested in efficiently handling  the complementary over-approximation problem, namely to obtain precise upper bounds on the error. This problem boils down to finding tight abstractions of linearities or non-linearities while being able to bound the resulting approximations in an efficient way.  
For computer programs consisting of linear operations, automatic error analysis can be obtained with well-studied optimization techniques based on SAT/SMT solvers~\cite{hgbk2012fmcad} and affine arithmetic~\cite{fluctuat}. However, non-linear operations are key to many interesting computational problems arising in physics, biology, controller implementations and global optimization. 
Recently, two promising frameworks have been designed to provide upper bounds for roundoff errors of nonlinear programs. The corresponding algorithms rely on Taylor-interval methods~\cite{fptaylor15}, implemented in the $\fptaylor$ tool, and on combining SMT with interval arithmetic~\cite{Darulova14Popl}, implemented in the $\rosa$ real compiler. 

%
The complexity of the mathematics underlying techniques for nonlinear reasoning, and the intricacies associated with constructing an efficient implementation, are such that a means for independent formal validation of results is particularly desirable.  
The $\rosa$ tool is based on theoretical results that should provide sound over-approximations of error bounds. While $\rosa$ relies on an SMT solver capable of generating unsatisfiability proof witnesses (thus allowing independent soundness checking), it does not  formally verify these certificates inside a proof assistant.
To the best of our knowledge, $\fptaylor$ and {\sc Gappa} are the only academic software tools that can produce formal proof certificates. For $\fptaylor$, this is based on the framework developed in~\cite{SolovyevH13} to verify nonlinear inequalities in $\hol$~\cite{hollight} using Taylor-interval methods. However, most of computation performed in the informal optimization procedure ends up being redone inside the $\hol$ proof assistant, yielding a formal verification which may be computationally demanding.

The aim of this work is to provide a formal framework to perform automated precision analysis of computer programs that manipulate finite-precision data using nonlinear operators. For such programs, guarantees can be provided with certified programming techniques.
Semidefinite programming (SDP) is relevant to a wide range of mathematical fields, including combinatorial optimization, control theory and matrix completion. In 2001, Lasserre introduced a hierarchy of SDP relaxations~\cite{Lasserre01moments} for approximating polynomial infima. Our method to bound the error is a decision procedure based on a specialized variant of the Lasserre hierarchy~\cite{Las06SparseSOS}. The procedure relies on SDP to provide sparse sum-of-squares decompositions of nonnegative polynomials. Our framework handles polynomial program analysis (involving the operations $+,\times,-$) as well as extensions to the more general class of semialgebraic and transcendental programs (involving $\sqrtsign, /, \min, \max, \arctan, \exp$), following the approximation scheme described in~\cite{Magron15sdp}.
\subsection{Overview of our Method}
\label{sec:overview}
We present an overview of our method and of the capabilities of related techniques, using an example.
Consider a program implementing the following polynomial expression $f$:
\begin{align*}
f(\x) := x_2 \times x_5 + x_3 \times x_6 - x_2 \times x_3  - x_5 \times x_6 \\
+ x_1 \times ( - x_1 +  x_2 +  x_3  - x_4 +  x_5 +  x_6) \,,
\end{align*}
where the six-variable vector $\x :=  (x_1, x_2, x_3, x_4, x_5, x_6)$ is the input of the program. For this example, assume that the set $\X$ of possible input values is a product of closed intervals: $\X = [4.00, 6.36]^6$.
This function $f$ together with the set $\X$ appear in many inequalities arising from the the proof of the Kepler Conjecture~\cite{Flyspeck06}, yielding challenging global optimization problems.


The polynomial expression $f$ is obtained by performing 15 basic operations (1 negation, 3 subtractions, 6 additions and 5 multiplications). 
When executing this program with a set of floating-point numbers $\hat{\x} :=  (\hat{x}_1, \hat{x}_2, \hat{x}_3, \hat{x}_4, \hat{x}_5, \hat{x}_6) \in \X$, one actually computes a floating-point result $\hat{f}$, where all operations $+, -, \times$ are replaced by the respectively associated floating-point operations $\oplus, \ominus, \otimes$.
The results of these operations comply with IEEE 754 standard arithmetic~\cite{IEEE} (see relevant background in Section~\ref{sec:fpbackground}). 
Here, for the sake of clarity, we do not consider real input variables but we do it later on while performing detailed comparison (see Section~\ref{sec:benchs}).
For instance, (in the absence of underflow) one can write $\hat{x}_2 \otimes \hat{x}_5 =  (x_2 \times x_5) (1 + e_1)$, by introducing an error variable $e_1$ such that $-\epsilon \leq e_1 \leq \epsilon$, where the bound $\epsilon$ is the machine precision (e.g.~$\epsilon = 2^{-24}$ for single precision). One would like to bound the absolute roundoff error $|r(\x, \e)| := | \hat{f}(\x, \e) - f (\x) |$ over  all possible input variables $\x \in \X$ and error variable  $e_1, \dots, e_{15} \in [-\epsilon, \epsilon]$. Let us define $\E := [-\epsilon, \epsilon]^{15}$ and $\K := \X \times \E$. Then our bound problem can be cast as finding the maximum $r^\star$ of $\mid r \mid$ over $\K$, yielding the following nonlinear optimization problem:
\begin{align}
\begin{split}
\label{eq:roptim}
r^\star := & \max_{(\x, \e) \in \K} | r(\x, \e) | \\
 = & \ \ \max \{-\min_{(\x, \e) \in \K} r(\x, \e), \max_{(\x, \e) \in \K} r(\x,\e)\} \enspace,
\end{split}
\end{align}
One can directly try to solve these two polynomial optimization problems using classical SDP relaxations~\cite{Lasserre01moments}.
As in~\cite{fptaylor15}, one can also decompose the error term $r$ as the sum of a term $l(\x,\e)$, which is affine w.r.t.~$\e$, and a nonlinear term $h(\x,\e) := r(\x,\e) - l(\x,\e)$. Then the triangular inequality yields:
\begin{equation}
\label{eq:lhoptim} 
r^\star \leq \max_{(\x, \e) \in \K} |l(\x, \e)| + \max_{(\x, \e) \in \K} |h(\x, \e)| \enspace. 
\end{equation}
It follows for this example that $l(\x,\e) = x_2 x_5 e_1 + x_3 x_6 e_2 +  (x_2 x_5 + x_3 x_6) e_3 + \dots + f(\x) e_{15} = \sum_{i=1}^{15} s_i(\x) e_i$, with $s_1(\x) := x_2 x_5, s_2(\x) := x_3 x_6, \dots, s_{15}(\x) := f(\x)$. The {\em Symbolic Taylor Expansions} method~\cite{fptaylor15} consists of using a simple branch and bound algorithm based on interval arithmetic to compute a rigorous interval enclosure of each polynomial $s_i$, $i = 1,\dots,15$, over $\X$ and finally obtain an upper bound of $|l| + |h|$ over $\K$. In contrast, our method uses sparse semidefinite relaxations for polynomial optimization (derived from \cite{Las06SparseSOS}) to bound $l$ and basic interval arithmetic as in~\cite{fptaylor15} to bound $|h|$ (i.e.~we use interval arithmetic to bound second-order error terms in the multivariate Taylor expansion of $r$ w.r.t.~$\e$).

The following comparison results have been obtained on an Intel Core i7-5600U CPU ($2.60\, $GHz). All execution times have been computed by averaging over five runs.
%
\begin{itemize}[noitemsep,nolistsep]
\item A direct attempt to solve the two polynomial problems occurring in Equation~\eqref{eq:roptim} fails as the SDP solver (in our case $\sdpa$~\cite{sdpa7}) runs out of memory. 
\item Using our method implemented in the $\realtofloat$ tool, one obtains an upper bound of $760 \epsilon$ for $|l| + |h|$ over $\K$ in $0.15$ seconds. This bound is provided together with a certificate which can be formally checked inside the $\coq$ proof assistant in $0.20$ seconds.
\item After normalizing the polynomial expression and using basic interval arithmetic, one obtains 8 times more quickly a coarser bound of $922 \epsilon$. 
\item Symbolic Taylor expansions implemented in $\fptaylor$ \cite{fptaylor15} provide a more precise bound of $721 \epsilon$, but the analysis time is 28 times slower than with our implementation. Formal verification of this bound inside the $\hol$ proof assistant takes $27.7$ seconds, which is 139 times slower than proof checking with $\realtofloat$ inside $\coq$. One can obtain an even more precise bound of $528 \epsilon$ (but 37 times slower than with our implementation) by turning on the improved rounding model of $\fptaylor$ and limiting the number of branch and bound iterations to 10000. The drawback of this bound is that it cannot be formally verified.
\item Finally, a sligthly coarser bound of $762 \epsilon$ is obtained with the $\rosa$ real compiler~\cite{Darulova14Popl}, but the analysis is 19 times slower than with our implementation and we cannot get formal verification of this bound.
\end{itemize}
%

\subsection{Related Works}
\label{sec:related}
%
SMT solvers allow analysis of programs with various semantics or specifications but are limited for the manipulation of problems involving nonlinear arithmetic. 
Several solvers, including {\sc Z3}~\cite{DeMoura08}, provide partial support for the IEEE floating-point standard~\cite{smtFPA2010}. They suffer from a lack of scalability when used for roundoff error analysis in isolation (as emphasized in~\cite{Darulova14Popl}), but can be integrated into existing frameworks, e.g.~{\sc FPhile}~\cite{PaganelliA13}. The procedure in~\cite{dReal13} can solve SMT problems over the real numbers, using interval constraint propagation, but has not yet been applied to quantification of roundoff error.

The $\rosa$ tool~\cite{Darulova14Popl} provides a way to compile functional {\sc Scala} programs involving semialgebraic functions and conditional statements.
The tool uses affine arithmetic to provide sound over-approximations of roundoff errors, allowing for generation of finite precision implementations which fulfill the required precision given as input by the user. This tool thus relies on abstract interpration but bounds of the affine expressions are provided through an optimization procedure based on SMT. In our case, we use the same rounding model but provide approximations which are affine w.r.t. the additional error variables and nonlinear w.r.t. the input variables. Instead of using SMT, we bound the resulting expressions with optimization techniques based on semidefinite programming.
%
Abstract interpretation~\cite{CousotCousot77} has been extensively used in the context of static analysis to provide sound over-approximations, called {\em abstractions}, of the sets of values taken by program variables. The effects of variable assignments, guards and conditional branching statements are handled with several domain specific operators (e.g. inclusion, meet and join). Well studied abstract domains include intervals~\cite{Moore62} as well as more complicated frameworks based on affine arithmetic~\cite{Stolfi03}, octogons~\cite{octogons}, zonotopes~\cite{Zonotope10}, polyhedra~\cite{polyhedra08}, interval polyhedra~\cite{IntervalPoly09}, some of them being implemented inside a tool called {\sc Apron}~\cite{Apron09}. Abstract domains provide sound over-approximations of program expressions, and allow upper bounds on roundoff error to be computed. 
The {\sc Gappa} tool~\cite{Daumas10} relies on interval abstract domains with an extension to affine domains~\cite{Linderman10}, to reason about roundoff errors.
As demonstrated in~\cite{fptaylor15}, the bounds obtained inside {\sc Gappa} are often coarser than other methods. Formal guarantees can be provided as {\sc Gappa} benefits from an interface with $\coq$ while making use of interval libraries~\cite{Melquiond201214} relying on formalized floating-points~\cite{BM11Flocq}. The static analysis commercial tool {\sc Fluctuat} (with a free academic version) relies on affine abstract domains~\cite{Blanchet03} and techniques which are very similar to the ones in $\rosa$, including interval subdivision. 
This tool does not perform optimization but uses forward computation to analyze floating-point programs written in C. 
Furthermore, {\sc Fluctuat} also has a procedure for discontinuity errors~\cite{Zonotope10}. The {\sc Gappa} and {\sc Fluctuat} tools use a different rounding model (also available as an option inside $\fptaylor$) based on a piecewise constant absolute error bound. This is more precise than the simple rounding model used in our framework but requires (possibly) extensive use of a branch and bound algorithm as each interval has to be subdivided in intervals $[2^n, 2^{n+1}]$ for several values of the integer $n$. 
In~\cite{fptaylor15}, the authors provide a table (Table~1) comparing relevant features of $\fptaylor$ with three other tools ($\rosa$, {\sc Gappa} and {\sc Fluctuat}), performing roundoff error estimation.
In a similar fashion, we summarize the main features related to our tool $\realtofloat$ and the same four above-mentioned tools used for our further benchmark comparisons w.r.t.~their expressiveness in Table~\ref{table:expressiveness}.

Computing sound bounds of nonlinear expressions is mandatory to perform formal analysis of finite precision implementations and can be performed with various optimization tools. 
In the polynomial case, alternative approaches to semidefinite relaxations are based on decomposition in the multivariate Bernstein basis. Formal verification of bounds obtained with this decomposition has been investigated by Mun\~oz and Narkawicz~\cite{MN13} in the PVS theorem prover. We are not aware of any work based on these techniques which can quantify roundoff errors. Another decomposition of nonnegative polynomials into SOS certificates consists in using the Krivine-Handelman~\cite{Krivine1964b,Handelman1988} representation and boils down to solving linear programming (LP) relaxations. In our case, we use a different representation, leading to solve SDP relaxations. The Krivine-Handelman representation has been used in~\cite{Boland10HGR} to compute roundoff error bounds. LP relaxations often provide coarser bounds than SDP relaxations and it has been proven in~\cite{lasserre2009moments} that generically finite convergence does not occur for convex problems, with the exception of the linear case. 
The work in ~\cite{Roux2015} focuses on formalization of roundoff errors bounds related to positive definiteness verification.
Branch and bound methods with Taylor models~\cite{Berz09} are not restricted to polynomial systems and have been formalized~\cite{SolovyevH13} to solve nonlinear inequalities occurring in the proof of Kepler Conjecture. Symbolic Taylor Expansions~\cite{fptaylor15} have been implemented in the $\fptaylor$ tool to compute formal bounds of roundoff errors for programs involving both polynomial and transcendental functions. 

\begin{table}[!t]
\begin{center}
\tbl{Comparison of roundoff error tools w.r.t.~expressiveness.\label{table:expressiveness}}{
\begin{tabular}{p{3.7cm}ccccc}
\hline
\multirow{1}{*}{Feature} & $\realtofloat$ & $\rosa$  & $\fptaylor$ & {\sc Gappa} & {\sc Fluctuat} \\
\hline
\multirow{1}{*}{{Basic FP operations/formats}} & $\surd$ & $\surd$ & $\surd$ & $\surd$ & $\surd$ \\
\multirow{1}{*}{{Special values ($\pm \infty$, NaN)}} &  &  &  & $\surd$ & $\surd$ \\
\multirow{1}{*}{{Improved rounding model}} &  & & $\surd$ & $\surd$ & $\surd$ \\
\multirow{1}{*}{{Input uncertainties}} & $\surd$ & $\surd$  & $\surd$ & $\surd$  & $\surd$ \\
\multirow{1}{*}{{Transcendental functions}} & $\surd$ &  & $\surd$  &  &  \\
\multirow{1}{*}{{Discontinuity errors}} & $\surd$ & $\surd$ &  &  & $\surd$ \\
\multirow{1}{*}{{Proof certificates}} & $\surd$ &  & $\surd$ & $\surd$  &  \\
\hline
\end{tabular}
}
\end{center}
\end{table}
\subsection{Contributions}
Our key contributions can be summarized as follows:
\begin{itemize}[noitemsep,nolistsep]
\item We present an optimization algorithm providing sound over-approximations for roundoff errors of floating-point nonlinear programs. 
This algorithm is based on sparse sums of squares programming~\cite{Las06SparseSOS}. In comparison with other methods, our algorithm allows us to obtain tighter upper bounds, while overcoming scalability and numerical issues inherent in SDP solvers~\cite{Todd01semidefiniteoptimization}. Our algorithm can currently handle  programs implementing polynomial functions, but also involving non-polynomial components, including either semialgebraic or transcendental operations (e.g. $/, \sqrtsign, \arctan, \exp$), as well as conditional statements.  Programs containing iterative or while loops are not currently supported.
\item Our framework is fully implemented in the $\realtofloat$ tool.  Among several features, the tool can optionally perform formal verification of roundoff error bounds for polynomial programs, inside the $\coq$ proof assistant~\cite{CoqProofAssistant}. The most recent software release of $\realtofloat$ provides $\ocaml$~\cite{OCaml} and $\coq$ libraries and is freely available.\footnote{\url{forge.ocamlcore.org/frs/?group_id=351}}
%
%
Our implementation tool is built on top of the $\nlcertify$ verification system~\cite{icms14}. Precision and efficiency of the tool are evaluated on several benchmarks coming from the existing literature. Numerical experiments demonstrate that our method competes well with recent approaches relying on Taylor-interval approximations~\cite{fptaylor15} or combining SMT solvers with affine arithmetic~\cite{Darulova14Popl}. We also compared our tool with {\sc Gappa}~\cite{Daumas10} and {\sc Fluctuat}~\cite{fluctuat}.
\end{itemize}
The paper is organized as follows.
In Section~\ref{sec:background}, we present mandatory background on roundoff errors due to finite precision arithmetic before describing our nonlinear program semantics (Section~\ref{sec:fpbackground}). Then we recall how to perform certified polynomial optimization based on semidefinite programming (Section~\ref{sec:sdpbackground}) and how to obtain formal bounds while checking the certificates inside the $\coq$ proof assistant (Section~\ref{sec:coqbackground}).
Section~\ref{sec:fpsdp} contains the main contribution of the paper, namely how to compute tight over-approximations for roundoff errors of nonlinear programs with sparse semidefinite relaxations.
Finally, Section~\ref{sec:benchs} is devoted to the evaluation of our nonlinear verification tool $\realtofloat$ on benchmarks arising from control systems, optimization, physics and biology, as well as comparisons with the tools $\fptaylor$, $\rosa$, {\sc Gappa} and {\sc Fluctuat}.
\section{PRELIMINARIES}
\label{sec:background}

\subsection{Program Semantics and Floating-point Numbers}
\label{sec:fpbackground}
%
We support conditional code without procedure calls or loops. Despite these restrictions, we can consider a wide range of nonlinear programs while assuming that  important numerical calculations can be expressed in a loop-free manner. 
Our programs are encoded in an ML-like language:
\begin{lstlisting}
let box_prog    $x_1 \dots x_n = [(a_1, b_1); \dots ; (a_n, b_n)]$;;
let obj_prog    $x_1 \dots x_n = [(f(\x), \epsilon_{\realtofloat})]$;;
let cstr_prog   $x_1 \dots x_n = [g_1 (\x); \dots; g_k(\x)]$;;
let uncert_prog $x_1 \dots x_n = [u_1; \dots; u_n]$;;
\end{lstlisting}
Here, the first line encodes interval floating-point bound constraints for input variables, namely $\x := (x_1, \dots, x_n) \in [a_1, b_1]\times \dots \times [a_n, b_n]$.
The second line provides the function $f(\x)$ as well as the total roundoff error bound $\epsilon_{\realtofloat}$.
Then, one encodes polynomial nonnegativity constraints over the input variables, namely $g_1(\x) \geq 0, \dots, g_k(\x) \geq 0$. Finally, the last line allows the user to specify a numerical constant $u_i$ to associate a given uncertainty to the variable $x_i$, for each $i= 1, \dots, n$.

The type of numerical constants is denoted by \code{C}. In our current implementation, the user can choose either 64 bit floating-point or arbitrary-size rational numbers. This type \code{C} is used for the terms $\epsilon_{\realtofloat}$, $u_1, \dots, u_n$, $a_1, \dots, a_n$, $b_1, \dots, b_n$.
The inductive type of polynomial expressions with coefficients in \code{C} is \code{pExprC} defined as follows:
\begin{lstlisting}
type pexprC = Pc of C | Px of positive 
| Psub of$\,$pexprC$\,$*$\,$pexprC | Pneg of pexprC 
| Padd of pexprC$\,$*$\,$pexprC 
| Pmul of pexprC$\,$*$\,$pexprC
\end{lstlisting}
The constructor \code{Px} takes a positive integer as argument to represent either an input or local variable.
The inductive type \code{nlexpr} of nonlinear expressions (such as $f(\x)$) is defined as follows:
\begin{lstlisting}
type nlexpr = 
| Pol of pexprC | Neg of nlexpr
| Add of nlexpr$\,$*$\,$nlexpr 
| Mul of nlexpr$\,$*$\,$nlexpr 
| Sub of nlexpr$\,$*$\,$nlexpr 
| Div of nlexpr$\,$*$\,$nlexpr | Sqrt of nlexpr 
| Transc of transc$\,$*$\,$nlexpr
| IfThenElse of pexprC$\,$*$\,$nlexpr$\,$*$\,$nlexpr
| Let of positive$\,$*$\,$nlexpr$\,$*$\,$nlexpr
\end{lstlisting}
The type \code{transc} corresponds to a {\em dictionary} $\setD$ of special functions. 
In our case  $\setD := \{\exp, \log, \cos, \sin, \tan, \arccos, \arcsin, \arctan \}$.
For instance, the term~\lstinline|Transc ($\exp$, $f(\x)$)| represents the program implementing $\exp(f(\x))$.

Given a polynomial expression $p$ and two nonlinear expressions $f$ and $g$, the term ~\lstinline|IfThenElse($p(\x)$, $f(\x)$, $g(\x)$)| represents the conditional program implementing the expression\footnote{Our general framework could theoretically handle nested if statements. However, our current implementation is limited to programs involving single top level conditional statements} \lstinline|if ($p(\x) \geq 0$) $f (\x)$ else $g (\x)$|.
The constructor \code{Let} allows us to define local variables in an ML fashion, e.g.~\lstinline|let $t_1 = 331.4 + 0.6 * T$ in $-t_1 * v /((t_1 + u) * (t_1 + u))$| (part of the \textit{doppler1} program considered in Section~\ref{sec:benchs}).
%
%

Finally, one obtains rounded nonlinear expressions using a recursive procedure~\lstinline|round|, defined according to Equation~\eqref{eq:roundbop} and Equation~\eqref{eq:roundtransc}. Rounded expressions are supported inside conditions. When an uncertainty $u_i$ is specified for an input variable $x_i$, the corresponding rounded expression is given by $x_i \, (1 + e)$, with $\mid e \mid \, \leq u_i$, the uncertainty $u_i$ being a relative error.

We adopt the standard practice~\cite{higham2002accuracy} to approximate a real number $x$ with its closest floating-point representation $\hat{x} = x (1 + e)$, with $|e|$ is less than the machine precision $\epsilon$. 
In the sequel, we neglect both overflow and denormal range values.
The operator $\hat{\cdot}$ is called the rounding operator and can be selected among rounding to nearest, rounding toward zero (resp.~$\pm\infty$). In the sequel, we assume rounding to nearest.
The scientific notation of a binary (resp.~decimal) floating-point number $\hat{x}$ is a triple $(s, sig, exp)$ consisting of a sign bit $s$, a {\em significand} $sig \in [1, 2)$ (resp.~$[1, 10)$) and an {\em exponent} $exp$, yielding numerical evaluation $(-1)^{s} \, sig \, 2^{exp}$ (resp.~$(-1)^{s} \, sig \, 10^{exp}$). 

The value of $\epsilon$ actually gives the upper bound on the relative floating-point error and is equal to $2^{-\prec}$, where $\prec$ is called the {\em precision}, referring to the number of significand bits used. For single precision floating-point, one has $\prec = 24$. For double (resp.~quadruple) precision, one has $\prec = 53$ (resp.~$\prec=113$). Let $\R$ denote the set of real numbers and $\F$ the set of binary floating-point numbers.

For each real-valued operation $\bop_\R \in \{+, -, \times, \slash \}$, the result of the corresponding floating-point operation $\bop_\F \in \{\oplus, \ominus, \otimes, \oslash \}$ satisfies the following when complying with IEEE 754 standard arithmetic~\cite{IEEE} (without overflow, underflow and denormal occurrences):
\begin{equation}
\label{eq:roundbop}
\bop_\F \, (\hat{x}, \hat{y}) = \bop_\R \, (\hat{x}, \hat{y}) \, (1 + e) \enspace, \quad \mid e \mid \leq \epsilon = 2^{-\prec} \enspace.
\end{equation}
Other operations include special functions taken from $\setD$, containing the unary functions
$\tan$, $\arctan$, $\cos$, $\arccos$, $\sin$, $\arcsin$, $\exp$, $\log$, $(\cdot)^{r}$ with $r\in \R\setminus\{0\}$. For $f_\R \in \setD$, the corresponding floating-point evaluation satisfies 
\begin{equation}
\label{eq:roundtransc}
f_\F (\hat{x}) = f_\R (\hat{x}) (1 + e) \enspace, \quad \mid e \mid \leq \epsilon (f_\R) \enspace.
\end{equation}
The value of the relative error bound $\epsilon (f_\R)$ differs from the machine precision $\epsilon$ in Equation~\eqref{eq:roundbop} and has to be properly adjusted on a per-operator basis. We refer the interested reader to~\cite{VerifCADTransc} for relative error bound verification of transcendental functions (see also~\cite{VerifHOLTransc} for formalization in $\hol$).
%


\subsection{SDP Relaxations for Polynomial Optimization}
\label{sec:sdpbackground}
The sums of squares method involves approximation of polynomial inequality constraints by sums of squares (SOS) equality constraints. Here we recall mandatory background about SOS. We apply this method in Section~\ref{sec:fpsdp} to solve the problems of Equation~\eqref{eq:roptim} when the nonlinear function $r$ is a polynomial. In the sequel, let us denote by $n$ the number of initial variables of the polynomial optimization problem and by $k$ the number of optimization constraints.
%
\subsubsection{Sums of squares certificates and SDP}
First we recall basic facts about generation of SOS certificates for polynomial optimization, using semidefinite programming, which can be found in texts such as~\cite{Lasserre01moments}.
Denote by $\R[\x]$ the vector space of polynomials and by $\R_{2 d}[\x]$ the restriction of $\R[\x]$ to polynomials of degree at most $2 d$. Let us define the set of SOS polynomials:
\begin{equation}
\label{eq:cone_sos}
\Sigma[\x] := \Bigl\{\sum_i q_i^2, \, \text{ with } q_i \in \R[\x] \Bigr\}\enspace,
\end{equation}
as well as its restriction $\Sigma_{2 d}[\x] := \Sigma[\x] \bigcap \R_{2 d}[\x]$ to polynomials of degree at most $2 d$. For instance, the following bivariate polynomial  $\sigma (\x) := 1 + (x_1^2 - x_2^2)^2$ lies in $\Sigma_4[\x] \subseteq \R_4[\x]$.

Optimization methods based on SOS use the implication $r \in \Sigma[\x] \implies \forall \x \in \R^n, \, r(\x) \geq 0$, i.e. the inclusion of $\Sigma[\x]$ in the set of nonnegative polynomials.

The underlying reason for using SOS polynomials is that optimizing over positive polynomials is NP Hard~\cite{laurent2009sums}. Thus, one would like to replace such positivity constraints by more tractable ones, and in particular the SOS decompositions admitted by positive polynomials provide a suitable alternative: when fixing the degree of such decompositions, the resulting relaxed problem becomes more tractable.

%
Given $r \in \R[\x]$, one considers the following polynomial minimization problem:
\begin{equation}
\label{eq:minpop}
r^*  :=  \inf_{\x \in \R^n} \, \{ \, r (\x) \, : \, \x \in \K \, \} \enspace,
\end{equation}
where the set of constraints $\K \subseteq \R^n$ is defined by
\[\K := \{ \x \in \R^{n} : g_1 (\x) \geq 0, \dots, g_k (\x) \geq 0\}\enspace,\]
for polynomial functions $g_1, \dots, g_k$. The set $\K$ is called a {\em basic semialgebraic} set. Membership of semialgebraic sets is ensured by satisfying conjunctions of polynomial nonnegativity constraints. 
\begin{remark}
\label{rk:arch}
 When the input variables satisfy interval constraints $\x \in [a_1, b_1] \times \dots \times [a_n, b_n]$ then one can easily show that there exists some integer $M > 0$ such that $M - \sum_{i=1}^n x_i^2 \geq 0$. 
In the sequel, we assume that this nonnegativity constraint appears explicitly in the definition of $\K$. Such an assumption is mandatory to prove the convergence of semidefinite relaxations recalled in Theorem~\ref{th:densesdp}.
\end{remark}
In general, the objective function $r$ and the set of constraints $\K$ can be nonconvex, which makes Problem~\eqref{eq:minpop} difficult to solve in practice. 
One can rewrite Problem~\eqref{eq:minpop} as the equivalent maximization problem:
\begin{equation}
\label{eq:maxpop}
r^*  :=  \sup_{\mu \in \R} \{ \, \mu \, : \, r (\x) - \mu \geq 0 \,, \ \forall \x \in \K \, \} \,.
\end{equation}
Now we outline how to handle the nonnegativity constraint $r - \mu \geq 0$.
Given a nonnegative polynomial $p \in \R[\x]$, the existence of an SOS decomposition $p = \sum_i q_i^2$ valid
over $\R^n$, is equivalent to the existence of a symmetric real matrix $\Q$, a solution of the following linear matrix feasibility problem:
\begin{align}
\label{eq:sdp}
r(\x) = \m_d(\x)^\intercal \, \Q \, \m_d(\x) \,, \quad \forall \x \in \R^n, \,
\end{align}
where $\m_d(\x) := (1, x_1, \dots, x_n, x_1^2,x_1 x_2,\dots, x_n^d)$ and the matrix $\Q$ has only nonnegative eigenvalues. Such a matrix $\Q$ is called {\em positive semidefinite}. The vector $\m_d$ (resp.~matrix $\Q$) has a size (resp.~dimension) equal to $s_n^d := \binom{n + d}{d}$. Problem~\eqref{eq:sdp} can be handled with semidefinite programming (SDP) solvers, such as {\sc Mosek}~\cite{mosek} or {\sc SDPA}~\cite{sdpa7} (see~\cite{Vandenberghe94SDP} for specific background about SDP). Then, one computes the ``LDL'' decomposition $\Q = \L^\intercal \D \L$ (a variant of the classical Cholesky decomposition), where $\L$ is a lower triangular matrix and $\D$ is a diagonal matrix. Finally, one obtains $r(\x) =  (\L \,
\m_d(\x))^\intercal \, \D \, (\L \, \m_d(\x)) = \sum_{i=0}^{s_n^d} q_i(\x)^2$. Such a decomposition is called a sums of squares (SOS) {\em certificate}.
\begin{example}
\label{ex:sdp}
Let us define $r(\x) := \frac{1}{4} + x_1^4 - 2 x_1^2 x_2^2 + x_2^4$. With $\m_2 (\x) = (1, x_1, x_2, x_1^2, x_1 x_2, x_2^2)$, one solves the linear matrix feasibility problem $r(\x) = \m_2 (\x)^\intercal \, \Q \, \m_2(\x)$. One can show that the solution writes $\Q = \L^\intercal \D \L$ for a $6 \times 6$ matrix $\L$ and a diagonal matrix $\D$ with entries $(\frac{1}{2},0,0,1,0,0)$, yielding the SOS decomposition: $r(\x) = (\frac{1}{2})^2 + (x_1^2 - x_2^2)^2$. This is enough to prove that $p$ is nonnegative.
\end{example}
\subsubsection{Dense SDP relaxations for polynomial optimization}
In order to solve our goal problem (Problem~\eqref{eq:roptim}), we are trying to solve Problem~\eqref{eq:minpop}, recast as
Problem~\eqref{eq:maxpop}. We first explain how to obtain tractable approximations of this difficult problem. Define $g_0 := 1$. The hierarchy of SDP relaxations developed by Lasserre \cite{Lasserre01moments} provides lower bounds of $r^*$, through solving the  optimization problems $(\P_d)$:
\[
(\P_d):\left\{			
\begin{array}{rlr}
p_d^\star := \sup\limits_{\sigma_j, \mu} & \mu \enspace, \\			 
\text{s.t.} &   r (\x) - \mu = \sum_{j = 0}^{k} \sigma_j(\x) g_j(\x) \,, \forall \x \,,\ \\
\\
& \mu\in \R \,, \sigma_j \in \Sigma[\x] \,, \quad \  \quad  \ j = 0,\dots,k \,, \\
\\
& \deg (\sigma_j g_j) \leq  2 d,             \quad \ \,  \qquad  j = 0,\dots,k \,.\\
\end{array} \right.
\]
One can solve $(\P_d)$ with SDP optimization to find a tuple $(\mu, \sigma_0, \dots, \sigma_k)$ which enables a proof that $g_1(\x) \geq 0 \wedge \dots  \wedge g_k(\x) \geq 0 \implies r (\x) - \mu \geq 0$.

The next theorem is a consequence of the assumption mentioned in Remark~\ref{rk:arch}.
\begin{theorem}[Lasserre~\cite{Lasserre01moments}]
\label{th:densesdp}
Let $p_d^{\star}$ be the optimal value of the SDP relaxation~$(\P_d)$.
Then, the sequence of optimal values $(p_d^\star)_{d \in \N}$ is nondecreasing and converges to $r^\star$.
\end{theorem}
The number of SDP variables (i.e.~the number of variables of the semidefinite relaxation~$(\P_d)$) grows polynomially with the integer $d$, called the {\em relaxation order}.
Indeed, at a fixed number of variables $n$, the relaxation~$(\P_d)$ involves $O((2 d)^{n})$ SDP
variables and $(k + 1)$ linear matrix inequalities (LMIs) of size
$O(d^n)$. When $d$ increases, then more accurate lower bounds of $r^\star$ can be obtained, at an increasing computational cost.
At a fixed $d$,  the relaxation $(\P_d)$ involves $O(n^{2d})$ SDP variables and $(d + 1)$ linear matrix inequalities (LMIs) of size
$O(n^{d})$.
\begin{example}
\label{ex:pop}
Consider the polynomial $f$ mentioned in Section~\ref{sec:intro}:
$f(\x) := x_2 x_5 + x_3 x_6 - x_2 x_3  - x_5 x_6 
+ x_1 ( - x_1 +  x_2 +  x_3  - x_4 +  x_5 +  x_6)$ and the set $\K := [4, 6.36]^6$. The set $\K$ can be equivalently rewritten as:  
\[
\K := \{\, \x \in \R^n \, : \, g_1 (\x) \geq 0, \dots, g_{7} (\x) \geq 0 \,\} 
\,,\]
with $g_i(\x) := (6.36 - x_i) (x_i - 4)$ for each $i = 1, \dots, 6$ and $g_7(\x) : = 243 - \sum_{i=0}^6 x_i^2$. Here the constant $M = 243$ is chosen so that $M \geq 6 \times 6.36^2$ and the assumption in Remark~\ref{rk:arch} is fulfilled.
The number of initial variables of the optimization problem $r^* := \inf_{\x \in \R^n} \, \{ \, f (\x) \, : \, \x \in \K \, \}$ is $n = 6$ and the number of optimization constraints is $k= 7$.
For $d=1$, the dense SDP relaxation~$(\P_1)$ involves $\binom{n + 2 d}{2 d} = \binom{6 + 2}{2} = 28$ variables and provides a lower bound $p_1^\star = 20.755$ for $r^*$. The dense SDP relaxation~$(\P_2)$ involves $\binom{6 + 4}{4} = 210$ variables and provides a tighter lower bound of $p_2^\star = 20.8608$ for $r^*$. 
\end{example}
\if{
There are several ways to decrease the size of the SDP problems. 
First, symmetries in SDP relaxations for polynomial optimization problems can be exploited to replace one SDP problem~$(\P_d)$ by
several smaller SDPs~\cite{Riener2013SymmetricSDP}. Notice that it is possible only if the multivariate polynomials of the initial problem are invariant under the action of a finite subgroup $G$ of the group $GL_{n}(\R)$. 
}\fi
\subsubsection{Exploiting sparsity}
Here we recall how to exploit the structured sparsity of the
problem to replace one SDP problem~$(\P_d)$ by an SDP problem~$(\S_d)$ of
size $O (\kappa^ {2 d})$ where $\kappa$ is the average size
of the maximal cliques of the correlation sparsity pattern (csp) of the polynomial
variables (see~\cite{Waki06SparseSOS,Las06SparseSOS} for more details). We now present these notions as well as the formulation of sparse SDP relaxations~$(\S_d)$.

We denote by $\N^n$ the set of $n$-tuple of nonnegative integers. The support of a polynomial $r(\x) := \sum_{\alphab \in \N^n} r_{\alphab} \x^{\alphab}$ is defined as $\suppf{r} := \{ \, \alphab \in \N^n \, : \, r_{\alphab} \neq 0 \, \}$. For instance the support of $r(\x) := \frac{1}{4} + x_1^4 - 2 x_1^2 x_2^2 + x_2^4$ is $\suppf{p} = \{ \, (0,0), (4, 0), (2,2), (0,4) \, \}$.

Let $F_j$ be the index set of variables which are involved in the polynomial $g_j$, for each $j=1, \dots, k$.
The correlative sparsity is represented by the 
$n \times n$ correlation sparsity pattern matrix (csp matrix) $\mathbf{R}$ defined by:
\begin{equation*}
\label{eq:csp}
\mathbf{R}(i, j) := \left \{
\begin{array}{ll}
  1 & \text{ if }  i = j \enspace, \\
  1 & \text{ if }  \exists \alphab \in \suppf{f} \text{ such that } \alpha_i, \alpha_j \geq 1 \,, \\
  1 & \text{ if }  \exists l \in \{1, \dots, k\} \text{ such that } i, j \in F_l  \,,\\
  0 & \text{otherwise .} 
\end{array} \right.
\end{equation*}

We define the undirected csp graph $G(N, E)$ with
 $N = \{ 1, \dots, n \}$ and $E = \{\{i, j\} : i, j \in N , \ i < j , \mathbf{R}(i, j) = 1 \}$. 
Then, let $C_1,\dots, C_m \subseteq N$ denote the maximal cliques of $G(N, E)$ and 
 define $n_j := \#C_j$, for each $j=1 ,\dots,m$.

\begin{remark}
\label{rk:sparsearch}
Assuming that the set $\K$ is as in Remark~\ref{rk:arch}, one replaces the constraint $M - \sum_{i=1}^n x_i^2 \geq 0$ by the $m$ redundant additional constraints:
\begin{equation}
\label{eq:assum_sos_sparse}
g_{k + j} := n_j M^2 - \sum_{i \in C_j} {x_i^2} \geq 0\,, \  j=1 ,\dots, m \,,
\end{equation}
set $k' = k + m$, define the compact semialgebraic set:
\[\K' := \{\, \x \in \R^n \, : \, g_1 (\x) \geq 0, \dots, g_{k'} (\x) \geq 0 \,\} \,,\]
and modify Problem~\eqref{eq:minpop} into the following optimization problem:
\begin{equation}
\label{eq:sparseminpop}
r^*  :=  \inf_{\x \in \R^n} \, \{ \, r (\x) \, : \, \x \in \K' \, \} \,.
\end{equation}
\end{remark}
For each $j=1 ,\dots,m$, we note $\R_{2 d}[\x, C_j]$ the set of polynomials of $\R_{2 d}[\x]$ which involve the variables $(x_i)_{i \in C_j}$. We denote $\Sigma [\x, C_j] := \Sigma [\x] \bigcap \R_{2 d}[\x, C_j]$. Similarly, we define $\Sigma [\x, F_j]$, for each $j=1, \dots, k'$.
The following program is the sparse variant of the SDP program $(\P_d)$:
\[
(\S_d):\left\{			
\begin{array}{rl}
r_d^\star := \sup\limits_{\mu, \sigma_j} & \mu\enspace, \\	 
\text{s.t.} & r (\x) - \mu = \sum_{j = 0}^{k'} \sigma_j(\x) g_j(\x) \,, \ \forall \x \,, \\
\\
& \mu\in \R \,,\  \sigma_0 \in \sum_{j = 1}^m \Sigma [\x, C_j] \,, \\
\\
& \sigma_j \in \Sigma[\x, F_j]  \,,\ j = 1,\dots,k' \,, \\
\\
& \deg (\sigma_j g_j) \leq 2 d  \,,\ j = 0,\dots,k' \,,
\end{array} \right.
\]
where $\sigma_0 \in \sum_{j = 1}^m \Sigma [\x, C_j]$ if and only if there exist $\sigma^1 \in \Sigma[\x, C_1], \dots, \sigma^m \in \Sigma[\x,C_m]$ such that $\sigma_0 (\x) = \sum_{j = 1}^m \sigma^j (\x)$, for all $\x \in \R^n$.

The number of SDP variables of the relaxation~$(\S_d)$ is $\sum_{j=1}^m \binom{n_j + 2 d}{2 d}$. At fixed $d$, it yields an SDP problem with $O(\kappa^{2d})$ variables, where $\kappa := \frac{1}{m} \sum_{j=1}^m n_j$ is the average size of the cliques $C_1, \dots, C_m$.
Moreover, the cliques $C_1, \dots, C_m$ satisfy the running intersection property: 
\begin{definition}[RIP]
\label{def:rip}
Let $m \in \N_0$  and $I_1, \dots, I_m$ be subsets of $\{1, \dots, n\}$. We say that $I_1, \dots, I_m$ satisfy the running intersection property (RIP) when for all $i=1, \dots, m$, there exists an integer $l < i$ such that $I_i \cap (\cup_{j < i} I_j) \subseteq I_l$.
\end{definition}
This RIP property together with the assumption mentioned in Remark~\ref{rk:sparsearch} allow us to state the sparse variant of Theorem~\ref{th:densesdp}:
%
\begin{theorem}[\protect{Lasserre~\cite[Theorem 3.6]{Las06SparseSOS}}]
\label{th:sparsesdp}
Let $r_d^{\star}$ be the optimal value of the sparse SDP relaxation~$(\S_d)$. Then the sequence $(r_d^{\star})_{d \in \N}$ is nondecreasing and converges to $r^\star$.
\end{theorem}
The interested reader can find more details in~\cite{Waki06SparseSOS} about additional ways to exploit sparsity in order to derive analogous sparse SDP relaxations.
We illustrate the benefits of the SDP relaxations~$(\S_d)$ with the following example:
\begin{example}
\label{ex:sparse}
Consider the polynomial $f$ mentioned in Section~\ref{sec:intro}:
$f(\x) := x_2 x_5 + x_3 x_6 - x_2 x_3  - x_5 x_6 
+ x_1 ( - x_1 +  x_2 +  x_3  - x_4 +  x_5 +  x_6)$.
Here, $n = 6, d = 2, N = \{1,\dots, 6 \}$. The $6 \times 6$ correlative sparsity matrix $\mathbf{R}$ is:
\[
\mathbf{R} = 
\begin{pmatrix}
  1 & 1 & 1 & 1 & 1 & 1 \\
  1 & 1 & 1 & 0 & 1 & 0 \\
  1 & 1 & 1 & 0 & 0 & 1 \\
  1 & 0 & 0 & 1 & 0 & 0 \\
  1 & 1 & 0 & 0 & 1 & 1 \\
  1 & 0 & 1 & 0 & 1 & 1 
 \end{pmatrix}
\]
The csp graph $G$ associated to $\mathbf{R}$ is depicted in Figure~\ref{fig:csp_deltax}. 
\begin{figure}[!t]	
\begin{center}
\includegraphics[scale=0.6]{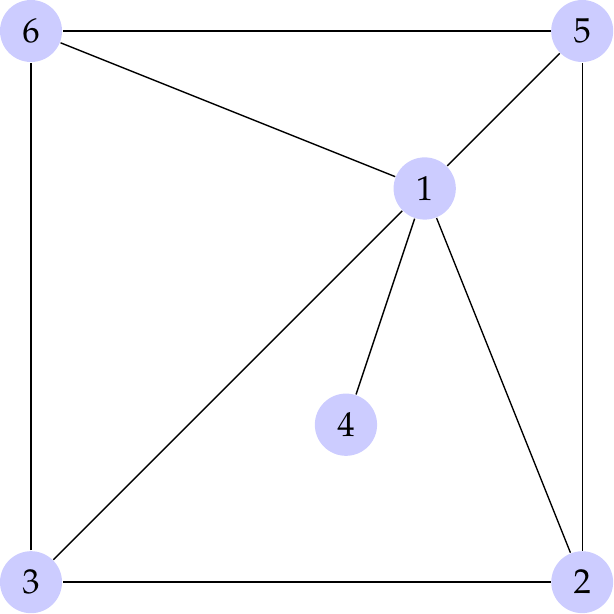}
\caption{Correlative sparsity pattern graph for the variables of $f$ from Example~\ref{ex:sparse}.}
\label{fig:csp_deltax}
\end{center}
\end{figure}
The maximal cliques of $G$ are $C_1 := \{1, 4\}$, $C_2 := \{1, 2, 3\}$, $C_3 := \{1, 2, 5\}$, $C_4 := \{1, 5, 6\}$ and $C_5 := \{1, 3, 6\}$. For $d=2$, the dense SDP relaxation~$(\P_2)$ involves $\binom{6 + 4}{4} = 210$ variables against $\binom{2 + 4}{4} + 4 \binom{3 + 4}{4} = 155$ for the sparse variant~$(\S_2)$. The dense SDP relaxation~$(\P_3)$ involves $924$ variables against $364$ for the sparse variant~$(\S_3)$. 
This difference becomes significant while considering that the time complexity of semidefinite programming is polynomial w.r.t. the number of variables with an exponent greater than 3 (see~\cite[Chapter 4]{BenTal01} for more details).
\end{example}
\subsection{Computer Proofs for Polynomial Optimization}
\label{sec:coqbackground}
Here, we briefly recall some existing features of the $\coq$ proof assistant to handle formal polynomial optimization, when using SDP relaxations.
The advantage of such relaxations is that they provide SOS certificates, which can be formally checked \textit{a posteriori}.
For more details on $\coq$, we recommend the
documentation available in~\cite{bertot2004interactive}.
Given a polynomial $r$ and a set of constraints $\K$, one can obtain a lower bound on $r$ by solving any instance of Problem~$(\P_d)$. Then, one can verify formally the correctness of the lower bound $r_d^\star$, using the SOS certificate output $\sigma_0, \dots, \sigma_k$. Indeed it is enough to prove the polynomial equality $r(\x) - r_d^\star = \sum_{j=0}^k \sigma_j(\x) g_j(\x)$ inside $\coq$. Such equalities can be efficiently proved using $\coq$'s ring tactic~\cite{ring05} via the mechanism of computational reflection~\cite{Boutin97usingreflection}. Any polynomial of type \code{pexprC} (see Section~\ref{sec:fpbackground}) can be normalized to a unique polynomial of type \code{polC} (see~\cite{ring05} for more details on the constructors of this type).
For the sake of clarity, let us consider the unconstrained case, i.e.~$\K = \R^n$. One encodes an SOS certificate $\sigma_0(\x) = \sum_{i=1}^m q_i^2$  with the sequence of polynomials $[q_1; \dots; q_m]$, each $q_i$ being of type \code{polC} . To prove the equality $r = \sigma_0$, our version of the ring tactic normalizes both $r$ and the sequence $[q_1; \dots; q_m]$ and compares the two normalization results. This mechanism is illustrated in Figure~\ref{fig:reflexion} with the polynomial $r(\x) := \frac{1}{4} + x_1^4 - 2 x_1^2 x_2^2 + x_2^4$ (see Example~\ref{ex:sdp}) being encoded by \code{r}  and the polynomials $1/2$ and $x_1^2 - x_2^2$ being encoded respectively by $\mathtt{q_1}$ and $\mathtt{q_2}$. 

\if{
In the general case, one applies a correctness lemma:
\begin{lstlisting}
Lemma correct_pop env $r$ cert_pop $ $:  
g_nonneg env g $\to$ checker g $\fpop$ $\mu_k^-$ cert_pop  $\eqcoq$ true $\to$ 
$\mu_k^-$ $\leq$ $\evalexpr{\fpop}$.
\end{lstlisting}
}\fi
\begin{figure}[!ht]
\centering
\includegraphics[scale=0.75]{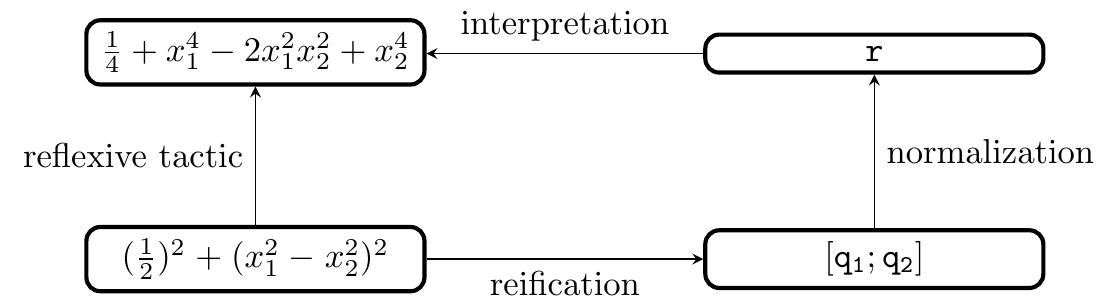}
\caption{An illustration of computational reflection.}	
\label{fig:reflexion}
\end{figure}
In the general case, this computational step is done through a \code{checker_sos} procedure which returns a Boolean value. If this value is true, one applies a correctness lemma, whose conclusion yields the nonnegativity of $r - r_d^\star$ over $\K$.
In practice, the SDP solvers are implemented in floating-point arithmetic, thus the above equality between $r - r_d^\star$ and the SOS certificate does not hold. However, following Remark~\ref{rk:arch}, each variable lies in a closed interval, thus one can bound the remainder polynomial $\epsilon(\x) := r(\x) - r_d^\star - \sum_{j=0}^k \sigma_j(\x) g_j(\x)$ using basic interval arithmetic, so that the lower bound $\epsilon^\star$ of $\epsilon$ yields the valid inequality: $\forall \x \in \K, r(\x) \geq r_d^\star + \epsilon^\star$.
For more explanation, we refer the interested reader to the formal framework~\cite[Section~2.3]{jfr14}. Note that this formal verification remains valid when considering the sparse variant~$(\S_d)$.
%
\section{GUARANTEED ROUNDOFF ERROR BOUNDS USING SDP RELAXATIONS}
\label{sec:fpsdp}
In this section, we present our new algorithm, relying on sparse SDP relaxations, to bound roundoff errors of nonlinear programs. After stating our general algorithm (Section~\ref{sec:transcsdp}), we detail how this procedure can handle polynomial programs (Section~\ref{sec:polsdp}). Extensions to the non-polynomial case, including conditional statements, are presented in Section~\ref{sec:nonpolsdp}. 
\subsection{The General Optimization Framework}
\label{sec:transcsdp}
Here we consider a given program that implements a nonlinear transcendental expression $f$ with input variables $\x$ satisfying a set of constraints $\X$. We assume that  $\X$ is included in a box (i.e.~a product of closed intervals) $[\a, \b] := [a_1, b_1] \times \dots \times [a_n, b_n]$ and that $\X$ is encoded as follows: 
\[ 
\X := \{\, \x \in \R^n \, : \, g_1 (\x) \geq 0, \dots, g_{k} (\x) \geq 0 \,\} \,,
\]
for polynomial functions $g_1, \dots, g_k$. 
Then, we denote by $\hat{f}(\x,\e)$ the rounded expression of $f$ after applying the ~\lstinline|round| procedure (see Section~\ref{sec:fpbackground}), introducing additional error variables $\e$.

The algorithm \code{bound}, depicted in Figure~\ref{alg:bound}, takes as input $\x$, $\X$, $f$, $\hat{f}$, $\e$ as well as the set $\E$ of bound constraints over $\e$. Here we assume that our program implementing $f$ does not involve conditional statements (this case will be discussed later in Section~\ref{sec:nonpolsdp}). For a given machine $\epsilon$, one has $\E := [-\epsilon, \epsilon]^m$, with $m$ being the number of error variables. This algorithm actually relies on the sparse SDP optimization procedure $(\S_d)$ (see Section~\ref{sec:sdpbackground} for more details), thus~\code{bound} also takes as input a relaxation order $d \in \N$. The algorithm provides as output an interval enclosure $I_d$ of the error $ \hat{f}(\x,\e) - f(\x)$ over $\K$. 
From this interval $I_d:= [\underline{f_d}, \overline{f_d}]$, one can compute $f_d := \max \{- \underline{f_d}, \overline{f_d} \}$, which is a sound upper bound of the maximal absolute error $r^\star := \max_{(\x,\e)\in \K} \mid \hat{f}(\x,\e) - f(\x) \mid $. 

\begin{figure}[!t]
\begin{algorithmic}[1]                    
\Require input variables $\x$, input constraints $\X$, nonlinear expression $f$, rounded expression $\hat{f}$, error variables $\e$, error constraints $\E$, relaxation order $d$
\Ensure interval enclosure $I_d$ of the error $\hat{f} - f$ over $\K := \X \times \E$
\State Define the absolute error $r(\x, \e) := \hat{f}(\x,\e) - f(\x)$ \label{line:r}

\State Compute $l(\x,\e) := r(\x, 0) + \sum_{j=1}^m \frac{\partial r(\x,\e)} {\partial e_j} (\x,0) \, e_j$ \label{line:l}

\State Define $h := r - l$ \label{line:h}

\State Compute bounds for $h$: $I^h := \iaboundfun{h}{\K}$ \label{line:iabound}
\State Compute bounds for $l$: $I_d^l := \sdpboundfun{l}{\K}{d}$  \label{line:sdpbound}
\State \Return $I_d := I_d^l + I^h$ 
\end{algorithmic}
\caption{\code{bound}: our algorithm to compute roundoff errors bounds of nonlinear programs.}
\label{alg:bound}
\end{figure}

After defining the absolute roundoff error $r := \hat{f} - f$ (Line~\lineref{line:r}), one decomposes $r$ as the sum of an expression $l$ which is affine w.r.t.~the error variable $\e$ and a remainder $h$. One way to obtain $l$ is to compute the vector of partial derivatives of $r$ w.r.t.~$\e$ evaluated at $(\x, 0)$ and finally to take the inner product of this vector and $\e$ (Line~\lineref{line:l}). 
Then, the idea is to compute a precise bound of $l$ and a coarse bound of $h$. 
The underlying reason is that $h$ involves error term products of degree greater than 2 (e.g.~$e_1 e_2$), yielding an interval enclosure $I^h$ of \textit{a priori} much smaller width, compared to the interval enclosure $I^l$ of $l$. 
%
%
One obtains $I^h$ using the procedure $\iabound$ implementing basic interval arithmetic (Line~\lineref{line:iabound}) to bound the remainder of the multivariate Taylor expansion of $r$ w.r.t.~$\e$, expressed as a combination of the second-error derivatives (similar as in~\cite{fptaylor15}).
The main algorithm presented in Figure~\ref{alg:bound} is very similar to the algorithm of $\fptaylor$~\cite{fptaylor15}, except that SDP based techniques are used instead of the global optimization procedure from~\cite{fptaylor15}. Note that overflow and denormal are neglected here but one could handle them, as in~\cite{fptaylor15}, by adding additional error variables and discarding the related terms using naive interval arithmetic.

%
\subsection{Polynomial Programs}
\label{sec:polsdp}
We first describe our $\sdpbound$ optimization algorithm when implementing polynomial programs. In this case, $\sdpbound$ calls an auxiliary procedure $\sdppoly$.
The bound of $l$ is provided through solving two sparse SDP instances of Problem~$(\S_d)$, at relaxation order $d$. We now give more explanation about the $\sdppoly$ procedure.

We can map each input variable $x_i$ to the integer $i$, for all $i=1,\dots,n$, as well as each error variable $e_j$ to $n+j$, for all $j=1,\dots,m$. Then, define the sets $C_1 := \{1,\dots,n,n+1\}, \dots, C_m := \{1,\dots,n,n+m\}$. Here, we take advantage of the correlation sparsity pattern of $l$ by using $m$ distinct sets of cardinality $n+1$ rather than a single one of cardinality $n+m$, i.e.~the total number of variables. 
After writing $l(\x,\e) = r(\x, 0) + \sum_{j=1}^m \frac{\partial r(\x,\e)} {\partial e_j} (\x,0) \, e_j$ and noticing that $r(\x,0) = \hat{f}(\x,0) - f(\x) = 0$, one can scale the optimization problems by writing 
\begin{align}
\label{eq:lscale}
l(\x,\e) = \sum_{j=1}^m s_j (\x) e_j = \epsilon \sum_{j=1}^m s_j (\x) \frac{e_j}{\epsilon} \,,
\end{align}
with $s_j(\x) := \frac{\partial r(\x,\e)} {\partial e_j} (\x,0)$, for all $j=1,\dots,m$. Replacing $\e$ by $\e/\epsilon$ leads to computing an interval enclosure of $l/\epsilon$ over $\K' := \X \times [-1, 1]^m$.
Recall that from Remark~\ref{rk:arch}, there exists an integer $M > 0$ such that $M - \sum_{i=1}^n x_i^2 \geq 0$, as the input variables satisfy box constraints.
Moreover, to fulfil the assumption of Remark~\ref{rk:sparsearch},  one encodes $\K'$ as follows: 
\begin{align*}
\K' := \{\, (\x,\e) \in \R^{n+m} \, : \, g_1 (\x) \geq 0, \dots, g_k(\x) \geq 0 \,, \\
g_{k+1}(\x,e_1) \geq 0, \dots, g_{k+m} (\x, e_m) \geq 0 \,\} \,,
\end{align*}
with $g_{k+j}(\x, e_j) := M + 1 -  \sum_{i=1}^n x_i^2 - e_j^2$, for all $j=1,\dots, m$. 
The index set of variables involved in $g_j$ is $F_j := N = \{1, \dots, n\}$ for all $j=1, \dots, k$. 
The index set of variables involved in $g_{k+j}$ is $F_{k+j} := C_j$ for all $j=1, \dots, m$. 

Then, one can compute a lower bound of the minimum of $l'(\x,\e) := l(\x, \e) / \epsilon = \sum_{j=1}^m s_j (\x) e_j$ over $\K'$ by solving the following optimization problem:
\if{
\begin{align*}
\label{eq:lscalesdp1}
\begin{split}
\left\{			
\begin{array}{rl}
\underline{l_d'} := \sup\limits_{\mu, \sigma_j} & \mu\enspace, \\	 
\text{s.t.} & l' - \mu = \sigma_0 + \sum_{j = 1}^{k+m} \sigma_j g_j \,, \\ 
\\
& \mu\in \R \,,\ \sigma_0 \in \sum_{j = 1}^m \Sigma [(\x, \e), C_j] \,, \\
\\
& \sigma_j \in \Sigma[(\x,\e), F_j] \,, \ j = 1,\dots,k+m \,, \\
\\
& \deg (\sigma_j g_j) \leq 2 d  \,, \ j = 1,\dots,k+m \,.
\end{array} \right.
\end{split}
\end{align*}
}\fi
\begin{align}
\label{eq:lscalesdp1}	
\begin{array}{rl}
\underline{l_d'} := \sup\limits_{\mu, \sigma_j} & \mu\enspace,  \\	 
\text{s.t.} & l' - \mu = \sigma_0 + \sum_{j = 1}^{k+m} \sigma_j g_j \,,  \\ 
& \mu\in \R \,,\ \sigma_0 \in \sum_{j = 1}^m \Sigma [(\x, \e), C_j] \,, \\
& \sigma_j \in \Sigma[(\x,\e), F_j] \,, \ j = 1,\dots,k+m \,, \\
& \deg (\sigma_j g_j) \leq 2 d  \,, \ j = 1,\dots,k+m \,. 
\end{array} 
\end{align}
A feasible solution of Problem~\eqref{eq:lscalesdp1} ensures the existence of $\sigma^1 \in \Sigma[(\x,e_1)], \dots, \sigma^m \in \Sigma[(\x,e_m)]$ such that $\sigma_0 = \sum_{j=0}^m \sigma^j$, allowing the following reformulation:
\begin{align}
\begin{split}
\label{eq:lscalesdp2}			
\begin{array}{rl}
\underline{l_d'} := \sup\limits_{\mu, \sigma_j} & \mu\enspace, \\	
\text{s.t.} & l' - \mu = \sum_{j=1}^m \sigma^j + \sum_{j = 1}^{k+m} \sigma_j g_j \,, \\
& \mu\in \R \,, \ \sigma_j \in \Sigma[\x] \,, \ j = 1,\dots,m \,, \\
& \sigma^j  \in \Sigma [(\x, e_j)] \,,  \deg (\sigma^j) \leq 2 d  \,, \ j = 1,\dots,m \,, \\
&  \quad \deg (\sigma_j g_j) \leq 2 d  \,, \ j = 1,\dots,k+m \,.
\end{array} 
\end{split}
\end{align}
An upper bound $\overline{l_d'}$ can be obtained by replacing $\sup$ with $\inf$ and $l' - \mu$ by $\mu - l'$ in Problem~\eqref{eq:lscalesdp2}.
Our optimization procedure $\sdppoly$ computes the lower bound $\underline{l_d'}$ as well as an upper bound $\overline{l_d'}$ of $l'$ over $\K'$ then returns the interval $I_d^l := [\epsilon \, \underline{l_d'}, \epsilon \, \overline{l_d'}] $, which is a sound enclosure of the values of $l$ over $\K$.
%

We emphasize two advantages of the decomposition $r := l + h$ and more precisely of the linear dependency of $l$ w.r.t.~$\e$: scalability and robustness to SDP numerical issues.
First, no computation is required to determine the correlation sparsity pattern of $l$, by comparison to the general case. Thus, it becomes much easier to handle the optimization of $l$ with the sparse SDP Problem~\eqref{eq:lscalesdp2} rather than with the corresponding instance of the dense relaxation~$(\P_d)$. While the latter involves $\binom{n + m+ 2 d}{2 d}$ SDP variables, the former involves only $m \, \binom{n + 1 + 2 d}{2 d}$ variables, ensuring the scalability of our framework.
In addition, the linear dependency of $l$ w.r.t.~$\e$ allows us to scale the error variables and optimize over a set of variables lying in $\K' := \X \times [-1, 1]^m$. It ensures that the range of input variables does not significantly differ from the range of error variables. This condition is mandatory while considering SDP relaxations because most SDP solvers (e.g.~{\sc Mosek}~\cite{mosek}) are implemented using double precision floating-point. It is impossible to optimize $l$ over $\K$ (rather than $l'$ over $\K'$) when the maximal value $\epsilon$ of error variables is less than $2^{-53}$, due to the fact that SDP solvers would treat each error variable term as 0, and consequently $l$ as the zero polynomial. Thus, this decomposition insures our framework against numerical issues related to finite-precision implementation of SDP solvers.
%

Let us define the interval enclosure $I^l := [\underline{l}, \overline{l}]$, with $\underline{l} := \inf_{(\x,\e) \in \K} l(\x,\e)$ and $\overline{l} := \sup_{(\x,\e) \in \K} l(\x,\e)$.
The next lemma states that one can approximate $I^l$ as closely as desired using the $\sdppoly$ procedure.
\begin{lemma}[Convergence of the $\sdppoly$ procedure]
\label{th:cvg_sdppoly}
Let $I_d^l$ be the interval enclosure returned by the procedure $\sdppolyfun{l}{\K}{d}$. The sequence $(I_d^l)_{d \in \N}$ converges to $I^l$.
\end{lemma}
%
\begin{proof}
It is sufficient to show the similar convergence result for $l' = l/\epsilon$, as it implies the convergence for $l$ by a scaling argument.
The sets $C_1,\dots, C_m$ satisfy the RIP property (see Definition~\ref{def:rip}). Moreover, the encoding of $\K'$ satisfies the assumption mentioned in Remark~\ref{rk:sparsearch}. Thus, Theorem~\ref{th:sparsesdp} implies that the sequence of lower bounds $(\underline{l_d'})_{d \in \N}$ converges to $\underline{l'} := \inf_{(\x,\e) \in \K'} l'(\x,\e)$. Similarly, the sequence of upper bounds converge to $\overline{l'}$, yielding the desired result.
\end{proof}
Lemma~\ref{th:cvg_sdppoly} guarantees asymptotic convergence to the exact enclosure of $l$ when the relaxation order $d$ tends to infinity. However, it is more reasonable in practice to keep this order as small as possible to obtain tractable SDP relaxations. Hence, we generically solve each instance of Problem~\eqref{eq:lscalesdp2} at the minimal relaxation order, that is $d_0 := \max \{\lceil \deg l / 2\rceil) , \max_{1 \leq j \leq k+m} \{ \lceil \deg (g_j) / 2\rceil) \} \}$. 
%

\subsection{Non-polynomial and Conditional Programs}
\label{sec:nonpolsdp}
%
Other classes of programs do not only involve polynomials but also semialgebraic and transcendental functions as well as conditional statements. Such programs are of particular interest as they often occur in real-world applications such as biology modeling, space control or global optimization. We present how the general optimization procedure $\sdpbound$ can be extended to these nonlinear programs.
%
\subsubsection{Semialgebraic programs}
Here we assume that the function $l$ is semialgebraic, that is it involves non-polynomial components such as divisions or square roots.
Following~\cite{LasPut10}, we explain how to transform the optimization problem $\inf_{(\x,\e) \in \K} l (\x, \e)$ into a polynomial optimization problem, then use the sparse SDP program~\eqref{eq:lscalesdp2}. One way to perform this reformulation consists of introducing lifting variables to represent non-polynomial operations.
We first illustrate the extension to semialgebraic programs with an example.
\begin{example}
Let us consider the program implementing the rational function $f : [0, 1] \to \R$ defined by $f(x_1) := \frac{x_1}{1 + x_1}$. Applying the rounding procedure (with machine $\epsilon$) yields $\hat{f}(x_1,\e) := \frac{x _1(1 + e_2)}{(1 + x_1)(1 + e_1)}$ and the decomposition $r(x_1, \e) := \hat{f}(x_1,\e) - f(x_1) = l(x_1,\e) + h(x_1,\e) = s_1 (x_1) e_1 + s_2 (x_1) e_2 + h(x_1,\e)$. One has $s_1(x_1) = \frac{\partial r(x_1,\e)} {\partial e_1} (x_1,0) = -\frac{x_1}{1 + x_1}$ and $s_2(x_1) = - s_1(x_1)$.

Let $\K := [0, 1] \times [-\epsilon, \epsilon]^2$. One introduces a lifting variable $x_2 := \frac{x_1}{1 + x_1}$ to handle the division operator and encode the equality constraint $p(\x) :=  x_2 (1 + x_1) - x_1 = 0$ with the two inequality constraints $p (\x) \geq 0$ and $-p(\x) \geq 0$. To ensure the compactness assumption, one bounds $x_2$ within $I := [0, 1/2]$, using basic interval arithmetic.

Let $\Kpol := \{(\x,\e) \in [0, 1] \times I \times [\epsilon, \epsilon]^2 : p(\x) \geq 0 \,,\  - p(\x) \geq 0 \}$. Then the rational optimization problem involving $l$ is equivalent to $\inf_{(\x,\e) \in \Kpol} x_2 (-e_1 + e_2)$, a polynomial optimization problem that we can handle with the $\sdppoly$ procedure, described in Section~\ref{sec:polsdp}.
\end{example}

In the semialgebraic case, $\sdpbound$ calls an auxiliary procedure $\sdpsa$. 
Given input variables $\y := (\x,\e)$, input constraints $\K := \X \times \E$ and a 
semialgebraic function $l$,  $\sdpsa$ first applies a recursive procedure $\lift$ 
which returns variables $\y\poly$, constraints $\K\poly$ and a polynomial $f\poly$ such that the 
interval enclosure $I^l$ of $l(\y)$ over $\K$ is equal to the interval enclosure of 
the polynomial $l\poly(\y\poly)$ over $\K\poly$. 
Calling $\sdpsa$ yields the interval enclosure $I^l_d := \sdppolyfun{l\poly}{\K\poly}{d}$. We detail the lifting procedure $\lift$ in Figure~\ref{alg:lift} for the constructors \code{Pol}(Line~\lineref{line:liftpol}), \code{Div} (Line~\lineref{line:liftdiv}) and \code{Sqrt} (Line~\lineref{line:liftsqrt}). 
The interval $I$ obtained through the $\iabound$ procedure (Line~\eqref{line:liftia}) allows us to constrain the additional variable $x$ to ensure the assumption of Remark~\ref{rk:sparsearch}.
For the sake of consistency, we omit the other cases (\code{Neg}, \code{Add}, \code{Mul} and \code{Sub}) where the procedure is straightforward. For a similar procedure in the context of global optimization, we refer the interested reader to~\cite[Chapter 2]{MagronPhD}.
\if{
\begin{figure}[!ht]
\begin{algorithmic}[1]                 
\Require input variables $\y$, input constraints $\K$, semialgebraic expression $f$
\Ensure variables $\y\poly$, constraints $\K\poly$, polynomial expression $f\poly$
     $I := \iaboundfun{f}{\K}$ \label{line:liftia}    
	\State \If{\lstinline|$f = \ $ Pol ($p$)|\label{line:liftpol} } $\y\poly := \y$, $\K\poly := \K$, $\f\poly := p$ \EndIf
	\State \If {\lstinline|$f = \ $ Div ($g$, $h$)| \label{line:liftdiv}}  {
	\State $\y_g, \K_g, g\poly := \liftfun{\y}{\K}{g}$ 
	\State $\y_h, \K_h, h\poly := \liftfun{\y}{\K}{h}$
    \State $\y\poly := (\y_g,\y_h,x)$ \hspace{1cm} $f\poly := x$ 
	\State $\K\poly := \{\y\poly \in \K_g \times \K_h \times I : x g\poly = f\poly \}$}
	\State \If {\lstinline|$f = \ $ Sqrt ($g$)| \label{line:liftsqrt} }{
	\State $\y_g, \K_g, g\poly := \liftfun{\y}{\K}{g}$ 
	\State $\y\poly := (\y_g,x)$ \hspace{1cm} $f\poly := x$ 
	\State $\K\poly := \{\y\poly \in \K_g \times I : x^2 = g\poly \}$}\\
\code{...}

\Return $\y\poly, \K\poly, f\poly$
\end{algorithmic}
\caption{\code{lift}: a recursive procedure to reduce semialgebraic problems to polynomial problems.}
\label{alg:lift}
\end{figure}
}\fi
\begin{figure}[!t]
\begin{algorithmic}[1]                 
\Require input variables $\y$, input constraints $\K$, semialgebraic expression $f$
\Ensure variables $\y\poly$, constraints $\K\poly$, polynomial expression $f\poly$
    \State $I := \iaboundfun{f}{\K}$ \label{line:liftia}
	\If {\lstinline|$f = \ $ Pol ($p$)|} \label{line:liftpol} $\y\poly := \y$, $\K\poly := \K$, $\f\poly := p$
	\ElsIf {\lstinline|$f = \ $ Div ($g$, $h$)|} \label{line:liftdiv} 
	\State $\y_g, \K_g, g\poly := \liftfun{\y}{\K}{g}$ 
	\State $\y_h, \K_h, h\poly := \liftfun{\y}{\K}{h}$
    \State $\y\poly := (\y_g,\y_h,x)$ \hspace{1cm} $f\poly := x$ 
	\State $\K\poly := \{\y\poly \in \K_g \times \K_h \times I : x h\poly = g\poly \}$
	\ElsIf {\lstinline|$f = \ $ Sqrt ($g$)|} \label{line:liftsqrt} 
	\State $\y_g, \K_g, g\poly := \liftfun{\y}{\K}{g}$ 
	\State $\y\poly := (\y_g,x)$ \hspace{1cm} $f\poly := x$ 
	\State $\K\poly := \{\y\poly \in \K_g \times I : x^2 = g\poly \}$\\
\code{...}
	\EndIf
\State \Return $\y\poly, \K\poly, f\poly$
\end{algorithmic}
\caption{\code{lift}: a recursive procedure to reduce semialgebraic problems to polynomial problems.}
\label{alg:lift}
\end{figure}
The set of variables $\y\poly$ can be decomposed as $(\x\poly, \e)$, where $\x\poly$ gathers input variables with lifting variables and has a cardinality equal to $n\poly$.
Then, one easily shows that the sets $\{1, \dots, n\poly, e_1\}$,$\dots$,$\{1, \dots, n\poly, e_m\}$ satisfy the RIP, thus ensuring to solve efficiently the corresponding instances of Problem~\eqref{eq:lscalesdp2}.
\subsubsection{Transcendental programs}
The above lifting procedure allows an exact representation of the graph of a semialgebraic function with polynomials involving additional (lifting) variables. We consider a procedure to approximate transcendental functions with semialgebraic functions.
Here we assume that the function $l$ is transcendental,~i.e.~involves univariate non-semialgebraic components such as $\exp$ or $\sin$. 
We use the method presented in~\cite{Magron15sdp}, based on maxplus approximation of semiconvex transcendental
functions by quadratic functions. This idea comes from
optimal control~\cite{mceneaney-livre} and was
developed further to represent the value function by
a ``maxplus linear combination'', which is a supremum of quadratic polynomials at given points $x_i$. 
Given a set of points $(x_i)$, we approximate from above and from below every transcendental function $f_{\R}$ 
by infima and suprema of finitely many quadratic polynomials $(f_{x_i}^-)$ and $(f_{x_i}^+)$. 
Hence, we reduce the problem to 
semialgebraic optimization problems. 
We can interpret this method in a geometrical way 
by thinking of it in terms of ``quadratic cuts'', since quadratic inequalities are added to
approximate the graph of a transcendental function.

For each univariate transcendental function $f_{\R}$ in our dictionary set $\setD$, one assumes that $f_{\R}$ is twice differentiable, so that the univariate function $g := f_{\R} + \frac{\gamma}{2} |\cdot|^2$ is convex on $I$ for large enough $\gamma > 0$ (for more details, see the reference~\cite{mceneaney-livre}). It follows that there exists a constant $\gamma \leq \sup_{x\in I} -f_{\R}''(x)$ such that for all $x_i \in I$:
\begin{align}
\begin{split}
\label{eq:maxplus}
\forall x \in I, \quad f_{\R} (x)  \geq f_{x_i}^-(x) \,,\\
\text{with } f_{x_i}^- :=  -\frac{\gamma}{2} (x-x_i)^2 +f_{\R}'(x_i) (x - x_i) + f_{\R} (x_i) \,,
\end{split}
\end{align}
implying that for all $x \in I$, $f_{\R} (x)  \geq \max_{x_i \in I} f_{x_i}^-(x)$. Similarly, one obtains an upper-approximation $\min_{x_i \in I} f_{x_i}^+(x)$.
Figure~\ref{fig:logexp} provides such approximations for the function $f_{\R}(x) := \log (1 + \exp(x))$ on the interval $I := [-8, 8]$.
\begin{figure}[!ht]
\begin{center}
\includegraphics[scale=0.7]{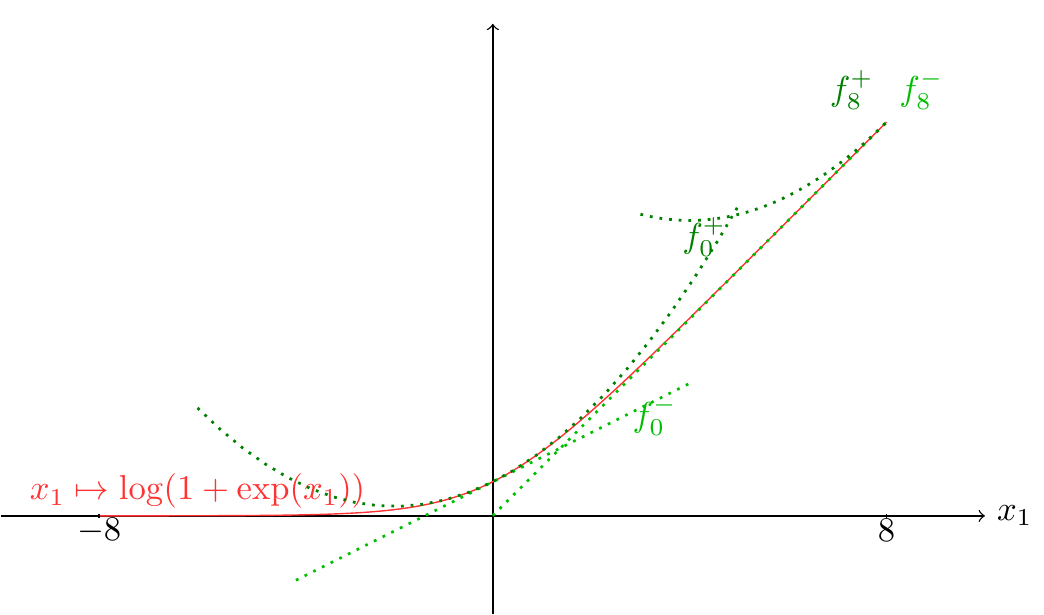}
\caption{Semialgebraic Approximations for $x \mapsto \log(1 + \exp(x))$: $\max \{ f_0^- (x), f_8^- (x)\} \leq  \log(1 + \exp (x)) \leq \min \{ f_0^+ (x), f_8^+ (x)\} $.}\label{fig:logexp}
\end{center}
\end{figure}

For transcendental programs, our procedure $\sdpbound$ calls the auxiliary procedure $\sdptransc$. Given input variables $(\x,\e)$, constraints $\K$ and a transcendental function $l$, $\sdptransc$ first computes a semialgebraic lower (resp.~upper)  approximation $l^-$ (resp.~$l^+$) of $l$ over $\K$. For more details in the context of global optimization, we refer the reader to~\cite{Magron15sdp}. Then, calling the procedure $\sdpsa$ allows us to get interval enclosures of $l^-$ as well as $l^+$.
We illustrate the procedure to handle transcendental programs with an example.
\begin{example}
\label{ex:logexp}
Let us consider the program implementing the transcendental function $f : [-8, 8] \to \R$ defined by $f(x_1) := \log (1 + \exp(x_1))$. Applying the rounding procedure  yields $\hat{f}(x_1,\e) := \log [(1 + \exp(x_1) (1 + e_1)) \, (1 + e_2)](1 + e_3)$. 
Here, $|e_2|$ is bounded by the machine $\epsilon$ while $|e_1|$ (resp.~$|e_3|$) is bounded with an adjusted absolute error $\epsilon_1 := \epsilon(\exp)$ (resp.~$\epsilon_3 := \epsilon(\log)$).
Let $\K:= [-8,8] \times [-\epsilon_1, \epsilon_1] \times [-\epsilon, \epsilon] \times [-\epsilon_3, \epsilon_3]$.

One obtains the decomposition $r(x_1, \e) := \hat{f}(x_1,\e) - f(x_1) = l(x_1,\e) + h(x_1,\e) = s_1 (x_1) e_1 + s_2 (x_1) e_2 + s_3 (x_1) e_3 + h(x_1, \e)$, with 
$s_1(x_1) = \frac{\exp(x_1)} {1 + \exp(x_1)}$, $s_2(x_1) = 1$ and $s_3(x_1) = \log (1 + \exp(x_1)) = f(x_1)$. Figure~\ref{fig:logexp} provides a lower approximation $s_3^- := \max\{f_0^-,f_8^-\}$ of $s_3$ as well as an upper approximation $s_3^+ := \min \{f_0^+,f_8^+\}$. One can get similar approximations $s_1^-$ and $s_1^+$ for $s_1$. 
One first obtains (coarse) interval enclosures $I_2 = \iaboundfun{s_1}{\K}$ and $I_3 = \iaboundfun{s_3}{\K}$ and one introduces extra variables $x_2 \in I_2$ and $x_3 \in I_3$ to represent $s_1$ and $s_3$ respectively.
Then, the interval enclosure of $l$ over $\K$ is equal to the interval enclosure of $l\sa(\x,\e) := x_2 e_1 + e_2 + x_3 e_3$ over the set $\K\sa:= \{(x_1,\e) \in \K \,, (x_2, x_3) \in I_2 \times I_3 \,, s_1^-(x_1) \leq x_2 \leq s_1^+(x_1) \,, s_3^-(x_1) \leq x_3 \leq s_3^+(x_1) \}$.
\end{example}
\subsubsection{Programs with conditionals}
Finally, we explain how to extend our bounding procedure to nonlinear programs involving conditionals through the recursive algorithm given in Figure~\ref{alg:bound_nlprog}. The overall procedure is very similar to the one implemented within the $\rosa$ tool~\cite[Section~7, Figure~6]{Darulova14Popl}.
The $\boundnlprog$ algorithm relies on the $\bound$ procedure (see Figure~\ref{alg:bound} in Section~\ref{sec:transcsdp}) to compute roundoff error bounds of programs implementing transcendental functions (Line~\lineref{line:noncnd}).
From Line~\lineref{line:cnd} to Line~\lineref{line:endcnd}, the algorithm handles the case when the program implements a function $f$ defined as follows:
\[   
f (\x) := 
     \begin{cases}
       g(\x) &\text{if } p(\x) \geq 0,\\
       h(\x) &\text{otherwise}.
     \end{cases}
\]
The first branch output is $g$ while the second one is $h$. More sophisticated conditionals, such as ``$p_1(x) \geq 0 \text{ or/and } p_2(x) \geq 0$'', are not handled at the moment but one could easily extend the current framework to do so.
%
\if{
\begin{figure}[!t]
\begin{algorithmic}[1]
\Require input variables $\x$, input constraints $\X$, nonlinear expression $f$, rounded expression $\hat{f}$, error variables $\e$, error constraints $\E$, relaxation order $d$
\Ensure interval enclosure $I_d$ of the error $\hat{f} - f$ over $\K := \X \times \E$
\SetAlgoNoLine
\State \eIf{\lstinline|$f = \ $ IfThenElse ($p, g, h$)| \label{line:cnd}} {
\State  $I_d^p := \boundfun{\x}{\X}{p}{\hat{p}}{\e}{\E}{d} = [\underline{p_d}, \overline{p_d}]$ \label{line:polcnd} 
\State $\X_1 := \{ \x \in \X : 0 \leq p(\x) \leq \overline{p_d} \}$ \label{line:X1} 
\State $\X_2 := \{ \x \in \X : \underline{p_d} \leq p(\x) \leq 0 \}$\label{line:X2} 
\State $\X_3 := \{ \x \in \X : 0 \leq p(\x) \}$\label{line:X3}
\State $\X_4 := \{ \x \in \X : p(\x) \leq 0 \}$\label{line:X4} 
\State $I_d^1 := \boundnlprogfun{\x}{\X_1}{g}{\hat{h}}{\e}{\E}{d}$\label{line:I1}
\State $I_d^2 := \boundnlprogfun{\x}{\X_2}{h}{\hat{g}}{\e}{\E}{d}$\label{line:I2} 
\State $I_d^3 := \boundnlprogfun{\x}{\X_3}{g}{\hat{g}}{\e}{\E}{d}$\label{line:I3} 
\State $I_d^4 := \boundnlprogfun{\x}{\X_4}{h}{\hat{h}}{\e}{\E}{d}$\label{line:I4} 
\State \Return $I_d := I_d^1 \cup I_d^2 \cup I_d^3 \cup I_d^4$ \label{line:endcnd}}
{\State \Return $I_d := \boundfun{\x}{\X}{f}{\hat{f}}{\e}{\E}{d}$ \label{line:noncnd}}
\end{algorithmic}
\caption{\code{bound_nlprog}: our algorithm to compute roundoff error bounds of  programs with conditional statements.}
\label{alg:bound_nlprog}
\end{figure}
}\fi
\begin{figure}[!t]
\begin{algorithmic}[1]
\Require input variables $\x$, input constraints $\X$, nonlinear expression $f$, rounded expression $\hat{f}$, error variables $\e$, error constraints $\E$, relaxation order $d$
\Ensure interval enclosure $I_d$ of the error $\hat{f} - f$ over $\K := \X \times \E$
\If{\lstinline|$f = \ $ IfThenElse ($p, g, h$)|} \label{line:cnd}
\State $I_d^p := \boundfun{\x}{\X}{p}{\hat{p}}{\e}{\E}{d} = [\underline{p_d}, \overline{p_d}]$ \label{line:polcnd}
\State $\X_1 := \{ \x \in \X : 0 \leq p(\x) \leq -\underline{p_d}\}$ \label{line:X1}
\State $\X_2 := \{ \x \in \X : - \overline{p_d} \leq p(\x) \leq 0 \}$\label{line:X2}
\State $\X_3 := \{ \x \in \X : 0 \leq p(\x) \}$\label{line:X3}
\State $\X_4 := \{ \x \in \X : p(\x) \leq 0 \}$\label{line:X4}
\State $I_d^1 := \boundnlprogfun{\x}{\X_1}{g}{\hat{h}}{\e}{\E}{d}$\label{line:I1}
\State $I_d^2 := \boundnlprogfun{\x}{\X_2}{h}{\hat{g}}{\e}{\E}{d}$\label{line:I2}
\State $I_d^3 := \boundnlprogfun{\x}{\X_3}{g}{\hat{g}}{\e}{\E}{d}$\label{line:I3}
\State $I_d^4 := \boundnlprogfun{\x}{\X_4}{h}{\hat{h}}{\e}{\E}{d}$\label{line:I4}
\State \Return $I_d := I_d^1 \cup I_d^2 \cup I_d^3 \cup I_d^4$ \label{line:endcnd}
\Else ~ \Return $I_d := \boundfun{\x}{\X}{f}{\hat{f}}{\e}{\E}{d}$ \label{line:noncnd}
\EndIf
\end{algorithmic}
\caption{\code{bound_nlprog}: our algorithm to compute roundoff error bounds of  programs with conditional statements.}
\label{alg:bound_nlprog}
\end{figure}
A preliminary step consists of computing the roundoff error enclosure $I_d^p := [\underline{p_d}, \overline{p_d}]$ (Line~\lineref{line:polcnd}) for the program implementing the polynomial $p$. 
Then the procedure computes bounds related to the discontinuity error of the branch, that is the maximal value between the four following errors: 
\begin{itemize}[noitemsep,nolistsep]
\item (Line~\lineref{line:I1}) the error obtained while computing the rounded result $\hat{h}$ of the second branch instead of computing the exact result $g$ of the first one, occurring for the set of variables $(\x,\e)$ such that $\hat{p}(\x,\e) \leq 0 \leq p(\x)$. For scalability and numerical issues, we consider an over-approximation $\X_1$ (Line~\lineref{line:X1}) of this set, where the variables $\x$ satisfy the relaxed constraints $0 \leq p(\x) \leq -\underline{p_d}$. Note that in this case, one has $l(\x, \e) = r(\x, 0) + \sum_{j=1}^m \frac{\partial r(\x,\e)} {\partial e_j} (\x,0) \, e_j$, with $r(\x, 0) =  h(\x) - g(\x) \neq 0$. In general, we expect the magnitude of the partial derivative sum to be very small compared to the one of $r(\x, 0)$.
\item (Line~\lineref{line:I2}) the error obtained while computing the rounded result $\hat{g}$ of the first branch instead of computing the exact result $h$ of the second one, occurring for the set of variables $(\x,\e)$ such that $p(\x) \leq 0 \leq \hat{p}(\x,\e)$. We also consider an over-approximation $\X_2$ (Line~\lineref{line:X2}), where the variables $\x$ satisfy the relaxed constraints $- \overline{p_d} \leq p(\x) \leq 0$.
\item (Line~\lineref{line:I3}) the roundoff error corresponding to the program implementation of $g$.
\item(Line~\lineref{line:I4}) the roundoff error corresponding to the program implementation of $h$.
\end{itemize}
\subsubsection{Simplification of error terms}
In addition, our algorithm \code{bound_nlprog} integrates several features to reduce the number of error variables. First, it memorizes all sub-expressions of the nonlinear expression tree to perform common sub-expressions elimination. 
We can also simplify error term products, thanks to the following lemma.
\begin{lemma}[\protect{Higham~\cite[Lemma 3.3]{higham2002accuracy}}]
\label{th:redproduct}
Let $\epsilon$ be the machine precision and assume that for a given integer $k$, one has $\epsilon < \frac{1}{k}$ and $\gamma_k := \frac{k \epsilon}{1 - k \epsilon}$. Then, for all $e_1, \dots, e_k \in [-\epsilon, \epsilon]$, there exists $\theta_k$ such that ${\prod_{i=1}^k (1 + e_i) = 1 + \theta_k}$ and $\mid \theta_k \mid \leq \gamma_k$.
\end{lemma}
Lemma~\ref{th:redproduct} implies that for any $k$ such that $\epsilon < \frac{1}{k}$, one has $\theta_k \leq (k + 1) \epsilon$. Our algorithm has an option to automatically derive safe over-approximations of the absolute roundoff error while introducing only one variable $e_1$ (bounded by $(k + 1) \epsilon$) instead of $k$ error variables $e_1, \dots, e_k$ (bounded by $\epsilon$). The cost of solving the corresponding optimization problem can be significantly reduced but it yields coarser error bounds.

\section{EXPERIMENTAL EVALUATION}
\label{sec:benchs}
Now, we present experimental results obtained by applying our general $\boundnlprog$ algorithm (see Section~\ref{sec:fpsdp}, Figure~\ref{alg:bound_nlprog}) to various examples coming from physics, biology, space control and optimization. 
The  $\boundnlprog$ algorithm is implemented in an open-source tool called $\realtofloat$, built in top of the $\nlcertify$ nonlinear verification package, relying on $\ocaml$ (Version $4.02.1$), $\coq$ (Version $8.4\text{pl}5$) and interfaced with the SDP solver $\sdpa$ (Version $7.3.9$). The SDP solver output numerical SOS certificates, which are converted into rational SOS using the {\sc Zarith} $\ocaml$ library (Version $1.2$), implementing arithmetic operations over arbitrary-precision integers.
For more details about the installation and usage of $\realtofloat$, we refer to the dedicated web-page\footnote{\url{http://nl-certify.forge.ocamlcore.org/real2float.html}} and the setup instructions.\footnote{see the \texttt{README.md} file in the top level directory}
All examples are displayed in Appendix A as the corresponding $\realtofloat$ input text files and satisfy our nonlinear program semantics (see Section~\ref{sec:fpbackground}). All results have been obtained on an Intel Core i7-5600U CPU ($2.60\, $GHz). Execution timings have been computed by averaging over five runs.

\subsection{Benchmark Presentation}
For each example, we compared the quality of the roundoff error bounds (Table~\ref{table:error}) and corresponding execution times (Table~\ref{table:cpu}) while running our tool $\realtofloat$, $\fptaylor$ (version from May $2016$~\cite{fptaylor15}), $\rosa$ (version from May $2014$ used in~\cite{Darulova14Popl}), {\sc Gappa} (version 1.2.0~\cite{Daumas10}) and {\sc Fluctuat} (version 3.1370~\cite{fluctuat}). 

To ensure fair comparison, our initial choice was to focus on tools providing certificates  and using the same rounding model ($\fptaylor$ or $\rosa$ which relies on an SMT solver theoretically able to output satisfiability certificates).
However, for the sake of completeness, we have also compared $\realtofloat$ with {\sc Gappa}~\cite{Daumas10} and  {\sc Fluctuat}~\cite{fluctuat}. For all tools, we use default parameters, we use the default number of subdivisions in {\sc Fluctuat} and for {\sc Gappa}, we provide the simplest user-provided hints we could think of. 
  
A head-to-head comparison is not straightforward here due to differences in the approaches: {\sc Gappa} uses an improved rounding model based on a piecewise constant absolute error bound (see Section~\ref{sec:related} for more details), and {\sc Fluctuat} does not produce output certificates. 
We also performed further experiments while turning on the improved rounding model of $\fptaylor$ (which is the same as in {\sc Gappa} and {\sc Fluctuat}). 

A given program implements a nonlinear function $f(\x)$, involving variables $\x$ lying in a set $\X$ contained in a box $[\a, \b]$.
Applying our rounding model on $f$ yields the nonlinear expression $\hat{f}(\x,\e)$, involving additional error variables $\e$ lying in a set $\E$. 

At a given semidefinite relaxation order $d$, our tool computes the upper bound $f_d$ of the absolute roundoff error $\mid f - \hat{f} \mid $ over $\K := \X \times \E$ and verifies that it is less than a requested number $\epsilon_{\realtofloat}^+$. As we keep the relaxation order $d$ as low as possible to ensure tractable SDP programs, it can happen that $f_d > \epsilon_{\realtofloat}^+$. 
The $\realtofloat$ tool has the default option to perform box subdivisions when the number of initial variables and maximal polynomial degree are both small. When the option is disabled, the solver does not perform subdivisions and outputs the error bound $f_d$. When enabled, we subdivide a randomly chosen interval of the box  $[\a, \b]$ in two halves to obtain two sub-sets $\X_1$ and $\X_2$, fulfilling $\X := \X_1 \cup \X_2$, and apply the $\boundnlprog$ algorithm on both sub-sets either until we succeed to certify that $\epsilon_{\realtofloat}^+$ is a sound upper bound of the roundoff error or until the maximal number of branch and bound iterations is reached. For each benchmark, an error bound $\epsilon_{\realtofloat}$ is automatically computed while setting $\epsilon_{\realtofloat}^+ = 0$.

The number $\epsilon_{\realtofloat}$ is compared with the upper bounds computed by two other tools implementing simple rounding models: $\fptaylor$, which relies on Taylor Symbolic expansions~\cite{fptaylor15}, and $\rosa$, which relies on SMT and affine arithmetic~\cite{Darulova14Popl}. 



For comparison purpose, we also executed each program using random inputs, following the approach used
in the $\rosa$ paper~\cite{Darulova14Popl}.  Specifically,
we executed each program on $10^7$ random inputs satisfying the input restrictions.  The results from these random
samples provide lower bounds on the absolute error. For consistency of comparison, the error bounds computed with $\fptaylor$ correspond to the procedure FPT.~(a) (see~\cite{fptaylor15}) using the same simplified rounding model as the one described in Equation~\ref{eq:roundbop}, also used in $\rosa$~\cite{Darulova14Popl}.
\begin{table*}[!ht]
\begin{center}
\tbl{Comparison results of upper and lower bounds for absolute roundoff errors among tools implementing either simple or advanced rounding model. For each model, results of the winning tool are emphasized using \textbf{bold fonts}.\label{table:error}}{
\begin{tabular}{lc|ccc|ccc|c}
\hline
& & \multicolumn{3}{c|}{Simple rounding} & \multicolumn{3}{c|}{Improved rounding} & \\
\multirow{1}{*}{Benchmark} &  \multirow{1}{*}{\texttt{id}} & \multirow{1}{*}{$\realtofloat$} & $\rosa$  & $\fptaylor$ & $\fptaylor$ & {\sc Gappa} & {\sc Fluctuat} &   \multirow{1}{*}{lower bound} \\
\hline  
\multicolumn{9}{l}{Programs involving polynomial functions}    \\
\hline
\multirow{1}{*}{\texttt{rigidBody1}} & \texttt{a}
& $5.33\text{e--}13$ & $5.08\text{e--}13$ & $\mathbf{3.87\textbf{e--}13}$ & $\mathbf{2.95\textbf{e--}13}$ & $\mathbf{2.95\textbf{e--}13}$ & $3.22\text{e--}13$  & $2.28\text{e--}13$\\
\multirow{1}{*}{\texttt{rigidBody2}} & \texttt{b}
& $6.48\text{e--}11$ & $6.48\text{e--}11$ & $\mathbf{5.24\textbf{e--}11}$ & $\mathbf{3.61\textbf{e--}11}$ & $\mathbf{3.61\textbf{e--}11}$ & $3.65\text{e--}11$  & $2.19\text{e--}11$\\
\multirow{1}{*}{\texttt{kepler0}} & \texttt{c}
& ${1.18\text{e--}13}$ & $1.16\text{e--}13$ & $\mathbf{1.05\text{e--}13}$ & $\mathbf{7.47\textbf{e--}14}$ & ${1.12\text{e--}13}$  & $1.26\text{e--}13$  & $2.23\text{e--}14$\\
\multirow{1}{*}{\texttt{kepler1}} & \texttt{d}
& $\mathbf{4.47\textbf{e--}13}$ & $6.49\text{e--}13$ & ${4.49\text{e--}13}$ & $\mathbf{2.87\textbf{e--}13}$ & $4.89\text{e--}13$ & $5.57\text{e--}13$   & $7.58\text{e--}14$\\
\multirow{1}{*}{\texttt{kepler2}} & \texttt{e}
& $\mathbf{2.09\textbf{e--}12}$ & $2.89\text{e--}12$ & ${2.10\text{e--}12}$ & $\mathbf{1.58\textbf{e--}12}$ & ${2.45\text{e--}12}$ & $2.90\text{e--}12$  & $3.03\text{e--}13$\\
\multirow{1}{*}{\texttt{sineTaylor}} & \texttt{f}
& $\mathbf{6.03\textbf{e--}16}$ & $9.56\text{e--}16$ & $6.75\text{e--}16$ & $\mathbf{4.44\textbf{e--}16}$ & $8.33\text{e--}02$ & $6.86\text{e--}16$  & $2.85\text{e--}16$\\
\multirow{1}{*}{\texttt{sineOrder3}} & \texttt{g}
& $1.19\text{e--}15$ & $1.11\text{e--}15$ & $\mathbf{9.97\textbf{e--}16}$ & ${7.95\text{e--}16}$ & $\mathbf{7.62\textbf{e--}16}$ & $1.03\text{e--}15$  & $3.34\text{e--}16$\\
\multirow{1}{*}{\texttt{sqroot}} & \texttt{h}
& $1.29\text{e--}15$ & $8.41\text{e--}16$ & $\mathbf{7.13\textbf{e--}16}$ & $\mathbf{5.02\textbf{e--}16}$ & ${5.37\text{e--}16}$ & $3.21\text{e--}13$   & $4.45\text{e--}16$\\
\multirow{1}{*}{\texttt{himmilbeau}} & \texttt{i}
& ${1.43\text{e--}12}$ & ${1.43\text{e--}12}$ & $\mathbf{1.32\textbf{e--}12}$ & $\mathbf{1.01\textbf{e--}12}$ & $\mathbf{1.01\textbf{e--}12}$ & $\mathbf{1.01\textbf{e--}12}$  & $1.47\text{e--}13$ \\
\hline
\multicolumn{9}{l}{Programs involving semialgebraic functions}    \\
\hline
\multirow{1}{*}{\texttt{doppler1}} & \texttt{j}
& $7.65\text{e--}12$ & $4.92\text{e--}13$ & $\mathbf{1.59\textbf{e--}13}$ & $\mathbf{1.29\textbf{e--}13}$ & $1.82\text{e--}13$ & ${1.34\text{e--}13}$  & $7.11\text{e--}14$\\
\multirow{1}{*}{\texttt{doppler2}} & \texttt{k}
& $1.57\text{e--}11$ & $1.29\text{e--}12$ & $\mathbf{2.90\textbf{e--}13}$ & $\mathbf{2.39\textbf{e--}13}$ & ${3.23\text{e--}13}$ & ${2.53\text{e--}13}$  & $1.14\text{e--}13$\\
\multirow{1}{*}{\texttt{doppler3}} & \texttt{l}
& $8.55\text{e--}12$ & $2.03\text{e--}13$ & $\mathbf{8.22\textbf{e--}14}$ & $\mathbf{6.96\textbf{e--}14}$ & $9.29\text{e--}14$ & ${7.36\text{e--}14}$  & $4.27\text{e--}14$\\
\multirow{1}{*}{\texttt{verhulst}} & \texttt{m}
& $4.67\text{e--}16$ & $6.82\text{e--}16$ & $\mathbf{3.53\textbf{e--}16}$ & $\mathbf{2.50\textbf{e--}16}$ & ${3.18\text{e--}16}$  & $4.84\text{e--}16$  & $2.23\text{e--}16$\\
\multirow{1}{*}{\texttt{carbonGas}} & \texttt{n}
& $2.21\text{e--}08$ & $4.64\text{e--}08$ & $\mathbf{1.23\textbf{e--}08}$ & $\mathbf{7.77\textbf{e--}09}$ & ${8.85\text{e--}09}$ & $1.86\text{e--}08$  & $4.11\text{e--}09$ \\
\multirow{1}{*}{\texttt{predPrey}} & \texttt{o}
& $2.52\text{e--}16$ & $2.94\text{e--}16$ & $\mathbf{1.89\textbf{e--}16}$ & $\mathbf{1.60\textbf{e--}16}$ & ${1.95\text{e--}16}$ & $2.45\text{e--}16$   & $1.47\text{e--}16$ \\
\multirow{1}{*}{\texttt{turbine1}} & \texttt{p}
& $2.45\text{e--}11$  & $1.25\text{e--}13$ & $\mathbf{2.33\textbf{e--}14}$ & $\mathbf{1.67\textbf{e--}14}$ & ${3.88\text{e--}14}$ & $6.09\text{e--}14$ & $1.07\text{e--}14$ \\
\multirow{1}{*}{\texttt{turbine2}} & \texttt{q}
& $2.08\text{e--}12$ & $1.76\text{e--}13$ & $\mathbf{3.14\textbf{e--}14}$ & $\mathbf{2.01\textbf{e--}14}$ & ${3.97\text{e--}14}$ &$8.96\text{e--}14$  & $1.43\text{e--}14$ \\
\multirow{1}{*}{\texttt{turbine3}} & \texttt{r}
& $1.71\text{e--}11$ & $8.50\text{e--}14$ & $\mathbf{1.70\textbf{e--}14}$ & $\mathbf{9.58\textbf{e--}15}$ & $9.96\text{e+}00$ & $4.90\text{e--}14$  & $5.33\text{e--}15$ \\
\multirow{1}{*}{\texttt{jetEngine}} & \texttt{s}
& $\text{OoM}$  & $1.62\text{e--}08$  & $\mathbf{1.50\textbf{e--}11}$  & $\mathbf{1.03\textbf{e--}11}$ & $1.32\text{e+}05$ & $1.82\text{e--}11$  & $5.46\text{e--}12$ \\
\hline
\multicolumn{9}{l}{Programs implementing polynomial functions with polynomial preconditions}\\
\hline
\multirow{1}{*}{\texttt{floudas2\_6}} & \texttt{t}
& $\mathbf{5.15\textbf{e--}13}$ & $5.87\text{e--}13$ & $7.88\text{e--}13$ & $\mathbf{5.94\textbf{e--}13}$ & $5.98\text{e--}13$ & $7.45\text{e--}13$  & $4.56\text{e--}14$ \\
\multirow{1}{*}{\texttt{floudas3\_3}} & \texttt{u}
& $5.81\text{e--}13$ & $\mathbf{4.05\textbf{e--}13}$ & $5.76\text{e--}13$ & $4.29\text{e--}13$ & $\mathbf{2.65\textbf{e--}13}$ & $4.32\text{e--}13$  & $1.48\text{e--}13$ \\
\multirow{1}{*}{\texttt{floudas3\_4}} & \texttt{v}
& $2.78\text{e--}15$ & ${2.56\text{e--}15}$ &  $\mathbf{2.23\textbf{e--}15}$ & $1.78\text{e--}15$ & $\mathbf{1.23\textbf{e--}15}$ & $2.23\text{e--}15$  & $3.80\text{e--}16$ \\
\multirow{1}{*}{\texttt{floudas4\_6}} & \texttt{w}
&  $1.82\text{e--}15$ & ${1.33\text{e--}15}$ & $\mathbf{1.23\textbf{e--}15}$ & $\mathbf{8.89\textbf{e--}16}$ & $\mathbf{8.89\textbf{e--}16}$ & $1.12\text{e--}15$  & $2.35\text{e--}16$ \\
\multirow{1}{*}{\texttt{floudas4\_7}} & \texttt{x}
& $\mathbf{1.06\textbf{e--}14}$ & $1.31\text{e--}14$ & $1.80\text{e--}14$ & $1.32\text{e--}14$ & $\mathbf{7.44\textbf{e--}15}$ & $1.71\text{e--}14$  & $7.31\text{e--}15$ \\
\hline
\multicolumn{9}{l}{Programs involving conditional statements}    \\
\hline
\multirow{1}{*}{\texttt{cav10}} & \texttt{y} 
& $\mathbf{2.91\textbf{e+}00}$ & $\mathbf{2.91\textbf{e+}00}$ & $-$ & $-$ & $-$ & $\mathbf{1.02\textbf{e+}02}$  & $2.90\text{e+}00$ \\
\multirow{1}{*}{\texttt{perin}} & \texttt{z} 
& $\mathbf{2.01\textbf{e+}00}$ & $\mathbf{2.01\textbf{e+}00}$ & $-$ & $-$ & $-$ & $\mathbf{4.91\textbf{e+}01}$  & $2.00\text{e+}00$ \\
\hline
\multicolumn{9}{l}{Programs implementing transcendental functions}\\
\hline
\multirow{1}{*}{\texttt{logexp}} & \texttt{\textalpha}
& ${2.52\text{e--}15}$ & $-$ & $\mathbf{2.07\textbf{e--}15}$ & $\mathbf{1.99\textbf{e--}15}$ & $-$ & $-$  & $1.19\text{e--}15$ \\
\multirow{1}{*}{\texttt{sphere}} & \texttt{\textbeta}
& ${1.53\text{e--}14}$ & $-$ & $\mathbf{1.29\textbf{e--}14}$ & $\mathbf{8.21\textbf{e--}15}$ & $-$ & $-$  & $5.05\text{e--}15$ \\
\multirow{1}{*}{\texttt{hartman3}} & \texttt{\textgamma}
& $2.99\text{e--}13$ & $-$  & $\mathbf{1.34\textbf{e--}14}$ & $\mathbf{4.97\textbf{e--}15}$ & $-$ & $-$  & $1.10\text{e--}15$ \\
\multirow{1}{*}{\texttt{hartman6}} & \texttt{\textdelta}
& $5.09\text{e--}13$ & $-$ & $\mathbf{2.55\textbf{e--}14}$ & $\mathbf{8.19\textbf{e--}15}$ & $-$ & $-$  & $2.20\text{e--}15$ \\
\hline
\end{tabular}
}
\end{center}
\end{table*}
For the sake of further presentation, we identify with a letter (from \code{a} to \code{z} and from \texttt{\textalpha} to \texttt{\textgamma}) each of the $\nbenchs$ nonlinear programs. The programs \code{a-b} and \code{f-s} are taken from the $\rosa$ paper~\cite{Darulova14Popl} and were used in the $\fptaylor$ paper~\cite{fptaylor15} as well:
%
\begin{itemize}[noitemsep,nolistsep]
\item The first $9$ programs implement polynomial functions: \code{a-b} come from physics, \code{c-e} are derived from expressions involved in the proof of Kepler Conjecture~\cite{Flyspeck06} and \code{f-h} implement polynomial approximations of the sine and square root functions. 
The program \code{i} is issued from the global optimization literature and implements the problem \textit{Himmilbeau} in~\cite{Ali05}.
\item The $10$ programs \code{j-s} implement semialgebraic functions: \code{j-l} and \code{p-s} come from physics, \code{m} and \code{o} from biology, \code{n} from control. All these programs are used to compare $\fptaylor$ and $\rosa$ in~\cite{fptaylor15}. 
\item The five programs \code{t-x} come from the global optimization literature and correspond respectively to Problem 2.6, 3.3, 3.4, 4.6 and 4.7 in~\cite{Floudas90}. We selected them as they typically involve nontrivial polynomial preconditions (i.e. $\X$ is not a simple box but rather a set defined with conjunction of nonlinear polynomial inequalities).
\item The two programs \code{y-z} involve conditional statements and come from the static analysis literature. They correspond to the two respective running examples of~\cite{Zonotope10} ({\sc Fluctuat}'s divergence error computation) and~\cite{Marechal14}. The first program \code{y} is used in the $\rosa$ paper~\cite{Darulova14Popl} for the analysis of branches discontinuity error.
\item The last four programs \texttt{\textalpha}-\texttt{\textgamma} involve transcendental functions. The two programs \texttt{\textalpha} and \texttt{\textbeta} are used in the $\fptaylor$ paper~\cite{fptaylor15} and correspond respectively to the program \code{logexp} (see Example~\ref{ex:logexp}) and the program \code{sphere} taken from NASA World Wind Java SDK~\cite{NASA}. The $2$ programs \texttt{\textgamma} and \texttt{\textdelta} respectively implement the functions coming from the optimization problems \textit{Hartman 3} and \textit{Hartman 6} in~\cite{Ali05}, involving both sums of exponential functions composed with quadratic polynomials.
\end{itemize}
\paragraph{Tool comparison settings} The five tools $\realtofloat$, $\rosa$, $\fptaylor$, {\sc Gappa} and {\sc Fluctuat} can handle programs with input variable uncertainties as well as any floating-point precision, but for the sake of conciseness, we only considered to compare their performance on programs implemented in double precision floating-point ($\epsilon = 2^{-53}$). By contrast with preliminary experiments presented in Section~\ref{sec:overview} where we considered floating-point input variables, we run each tool by considering all input variables as real variables, thus we apply the rounding operation to all of them.
For the programs involving transcendental functions, we followed the same procedure as in $\fptaylor$ while adjusting the precision $\epsilon \,  (f_{\R}) = 1.5 \epsilon$ for each special function $f_{\R} \in \{\sin, \cos, \log, \exp \}$. Each univariate transcendental function is approximated from below (resp.~from above) using suprema (resp.~infima) of linear or quadratic polynomials (see Example~\ref{ex:logexp} for the case of program \code{logexp}). 
%
\subsection{Comparison Results}
Comparison results for error bound computation are presented in Table~\ref{table:error}. 
Among the tools using the simple rounding model ($\realtofloat$, $\rosa$ and $\fptaylor$), our $\realtofloat$ tool computes the tightest upper bounds for $7$ (resp.~$4$) out of $\nbenchs$ benchmarks when comparing with $\rosa$ and $\fptaylor$ (resp.~all tools). For all programs \code{j-s} involving semialgebraic functions and the four programs $\mathtt{\alphab}$-$\mathtt{\deltab}$ involving transcendental functions, $\fptaylor$ computes the tightest bounds, when comparing with $\realtofloat$ and $\rosa$.\footnote{The running execution times of $\rosa$ may change with more recent versions}
One current limitation of $\realtofloat$ is its limited ability to manipulate symbolic expressions, e.g. computing rational function derivatives or yielding reduction to the same denominator. In particular, the analysis of program \code{s} aborted after running out of memory (meaning of the symbol OoM).
The $\fptaylor$ tool is better suited to handle programs that exhibit such rational functions and also includes an interface with the \textsc{Maxima} computer algebra system~\cite{maxima} to perform symbolic simplifications.

We mention that the interested reader can find more detailed experimental comparisons between the three tools implementing improved rounding models ($\fptaylor$, {\sc Fluctuat} and {\sc Gappa}) in~\cite[Section 5.2]{fptaylor15}. 
Note that $\fptaylor$ provides the tightest bounds for $24$ out of $\nbenchs$ benchmarks. 
The {\sc Gappa} software provides the tightest bounds for $8$ out of $\nbenchs$ benchmarks while being almost always faster than other tools.
As emphasized in~\cite{fptaylor15}, {\sc Gappa} can sometimes compute more precise bounds with more advanced user-provided hints.
Note that {\sc Fluctuat} provides  tighter bounds than {\sc Gappa} for most rational functions while being always slower. 
It would be worth implementing the same improved rounding model in $\realtofloat$ to perform numerical comparisons.


To the best of our knowledge, $\realtofloat$ is the only academic tool which is able to handle the general class of programs involving either transcendental functions or conditional statements. The $\fptaylor$ (resp.~$\rosa$) tool does not currently handle conditionals (resp.~transcendental functions), as meant by the symbol $-$ in the corresponding column entries. 
However, an interface bridging the $\fptaylor$ and $\rosa$ tools would provide each other with the relevant missing features. 
These error bound comparison results together with their corresponding execution timings (given in Table~\ref{table:cpu}) are used to plot the data points shown in Figure~\ref{fig:timebound}.
\begin{table}[!ht]
\begin{center}
\tbl{Comparison of execution times (in seconds) for absolute roundoff error bounds among tools implementing either simple or advanced rounding model. For each model, the winner results are emphasized using \textbf{bold fonts}.\label{table:cpu}}{
\begin{tabular}{p{2.3cm}c|ccc|ccc}
\hline
& & \multicolumn{3}{c|}{Simple rounding} & \multicolumn{3}{c}{Improved rounding} \\
\multirow{1}{*}{Benchmark} & \texttt{id} & $\realtofloat$ & $\rosa$  & $\fptaylor$ & $\fptaylor$ & {\sc Gappa} & {\sc Fluctuat} \\
\hline

\multirow{1}{*}{\texttt{rigidBody1}} & \texttt{a} &
${0.58}$ & $\mathbf{0.13}$ & $1.84$ & $0.41$ & $\mathbf{0.10}$ & ${0.29}$ \\
\multirow{1}{*}{\texttt{rigidBody2}} & \texttt{b} &
$\mathbf{0.26}$ & $2.17$ & $3.01$ & ${0.46}$ & $\mathbf{0.15}$ & $0.52$ \\
\multirow{1}{*}{\texttt{kepler0}} & \texttt{c} & 
$\mathbf{0.22}$ & $3.78$ & $4.93$ & $6.96$ & $\mathbf{0.44}$ & ${0.70}$\\
\multirow{1}{*}{\texttt{kepler1}} & \texttt{d} &
${17.6}$ & $63.1$ & $\mathbf{9.33}$ & $4.90$ &  ${0.72}$ & $\mathbf{0.37}$ \\
\multirow{1}{*}{\texttt{kepler2}} & \texttt{e} &
$\mathbf{16.5}$ & $106$ & $19.1$ & $13.8$ & $\mathbf{1.58}$ & ${2.04}$ \\
\multirow{1}{*}{\texttt{sineTaylor}} & \texttt{f} &
$\mathbf{1.05}$ & $3.50$ & $2.91$ & $\mathbf{0.11}$ & ${0.16}$ & ${0.55}$ \\
\multirow{1}{*}{\texttt{sineOrder3}} & \texttt{g} &
$\mathbf{0.40}$ & $0.48$ & $1.90$ & $0.40$ & $\mathbf{0.06}$ & ${0.22}$ \\
\multirow{1}{*}{\texttt{sqroot}} & \texttt{h} &
$\mathbf{0.14}$ & $0.77$ & $2.70$ & $0.44$ & $\mathbf{0.19}$ &  ${0.40}$ \\
\multirow{1}{*}{\texttt{himmilbeau}} & \texttt{i} &
$\mathbf{0.20}$ & $2.51$ & $3.28$ & ${0.49}$ & $\mathbf{0.09}$ & $1.00$ \\
\multirow{1}{*}{\texttt{doppler1}} & \texttt{j} &
$6.80$ & ${6.35}$ & $\mathbf{6.13}$ & $1.34$ & $\mathbf{0.08}$ & ${0.77}$\\
\multirow{1}{*}{\texttt{doppler2}} & \texttt{k} &
$6.96$ & $\mathbf{6.54}$ & $6.88$ & $1.57$ & $\mathbf{0.08}$ & ${0.79}$ \\
\multirow{1}{*}{\texttt{doppler3}} & \texttt{l} &
$6.84$ & $\mathbf{6.37}$ & $9.13$ & $1.50$ & $\mathbf{0.07}$ &  ${0.78}$ \\
\multirow{1}{*}{\texttt{verhulst}} & \texttt{m} & 
$\mathbf{0.51}$ & $1.36$ & $1.37$ & $0.40$ & $\mathbf{0.05}$ & ${0.25}$ \\
\multirow{1}{*}{\texttt{carbonGas}} & \texttt{n} &
$\mathbf{0.83}$ & $6.59$ & $3.73$ & ${0.52}$ & $\mathbf{0.15}$ & ${2.03}$ \\
\multirow{1}{*}{\texttt{predPrey}} & \texttt{o} &
$\mathbf{0.87}$ & $4.12$ & $1.78$ & ${0.48}$ & $\mathbf{0.04}$ & $5.14$ \\
\multirow{1}{*}{\texttt{turbine1}} & \texttt{p} &
$72.2$ & $\mathbf{3.09}$ & $4.38$ & ${0.53}$ & $\mathbf{0.21}$ & $5.79$ \\
 \multirow{1}{*}{\texttt{turbine2}} & \texttt{q} &
$4.72$ & $7.75$ & $\mathbf{3.25}$ & ${0.60}$ & $\mathbf{0.12}$ & $4.76$ \\
 \multirow{1}{*}{\texttt{turbine3}} & \texttt{r} &
$74.5$ & $4.57$ & $\mathbf{3.46}$ & ${0.61}$ & $\mathbf{0.20}$ & $5.84$ \\
 \multirow{1}{*}{\texttt{jetEngine}} & \texttt{s} &
$\text{OoM}$  & $125$ & $\mathbf{9.79}$ & ${3.06}$ & $\mathbf{0.31}$ & $31.2$\\
\multirow{1}{*}{\texttt{floudas2\_6}} & \texttt{t} &
$\mathbf{2.49}$  & $159$ & $15.9$ & ${14.3}$ & $\mathbf{2.35}$ & $26.8$\\
\multirow{1}{*}{\texttt{floudas3\_3}} & \texttt{u} &
$\mathbf{0.45}$ & $13.9$ & $5.64$ & $13.9$ & $\mathbf{0.76}$ & $6.41$\\
\multirow{1}{*}{\texttt{floudas3\_4}} & \texttt{v} &
$\mathbf{0.09}$ & ${0.49}$ & $1.47$ & $1.27$ & $\mathbf{0.07}$ & $1.01$\\
\multirow{1}{*}{\texttt{floudas4\_6}} & \texttt{w} &
$\mathbf{0.07}$ & $1.20$ & $0.91$ & $0.37$ & $\mathbf{0.05}$ & $1.69$\\
\multirow{1}{*}{\texttt{floudas4\_7}} & \texttt{x} &
$\mathbf{0.13}$ & $21.8$ & $1.64$ & $0.39$ & $\mathbf{0.07}$ & $0.53$\\
\multirow{1}{*}{\texttt{cav10}} & \texttt{y} &
$\mathbf{0.23}$ & $0.59$ & $-$ & $-$ & $-$ & $1.26$ \\
\multirow{1}{*}{\texttt{perin}} & \texttt{z} &
$\mathbf{0.49}$ & $2.74$ & $-$ & $-$ & $-$ & $1.19$\\
\multirow{1}{*}{\texttt{logexp}} & \texttt{\textalpha} &
$\mathbf{1.05}$ & $-$ & ${1.10}$ & $\mathbf{0.39}$ & $-$ & $-$\\
\multirow{1}{*}{\texttt{sphere}} &  \texttt{\textbeta} &
$\mathbf{0.05}$ & $-$ & $2.04$ & $\mathbf{3.69}$ & $-$ & $-$\\
\multirow{1}{*}{\texttt{hartman3}} & \texttt{\textgamma} &
$\mathbf{2.02}$ & $-$ & $32.5$ & $\mathbf{27.8}$ & $-$ & $-$\\
\multirow{1}{*}{\texttt{hartman6}} & \texttt{\textdelta} &
$\mathbf{119}$ & $-$ & $364$ & $\mathbf{259}$ & $-$ & $-$\\
\hline
\end{tabular}
}
\end{center}
\end{table}
For each benchmark identified by \code{id}, let $t_{\realtofloat}$ (in 3rd column of Table~\ref{table:cpu}) refer to the execution time of $\realtofloat$ to obtain the corresponding upper bound $\epsilon_{\realtofloat}$ (in 3rd column of Table~\ref{table:error}).

Now, let us define the execution times $t_{\rosa}$, $t_{\fptaylor}$ and the corresponding error bounds $\epsilon_{\rosa}$, $\epsilon_{\fptaylor}$. Then the x-axis coordinate of the point \circled{id} (resp.~\squared{id}) displayed in Figure~\ref{fig:timeboundrosa} (resp.~\ref{fig:timeboundfptaylor}) corresponds to the logarithm of the ratio between the execution time of $\rosa$ (resp.~$\fptaylor$) and $\realtofloat$, i.e.~$\log \frac{t_{\rosa}}{t_{\realtofloat}}$ (resp.~$\log \frac{t_{\fptaylor}}{t_{\realtofloat}}$). Similarly, the y-axis coordinate of the point \circled{id} (resp.~\squared{id}) is $\log \frac{\epsilon_{\rosa}}{\epsilon_{\realtofloat}}$ (resp.~$\log \frac{\epsilon_{\fptaylor}}{\epsilon_{\realtofloat}}$). 

\begin{figure}[!ht]
\begin{center}
\subfigure[Comparison with $\rosa$.]{
\includegraphics[scale=0.53]{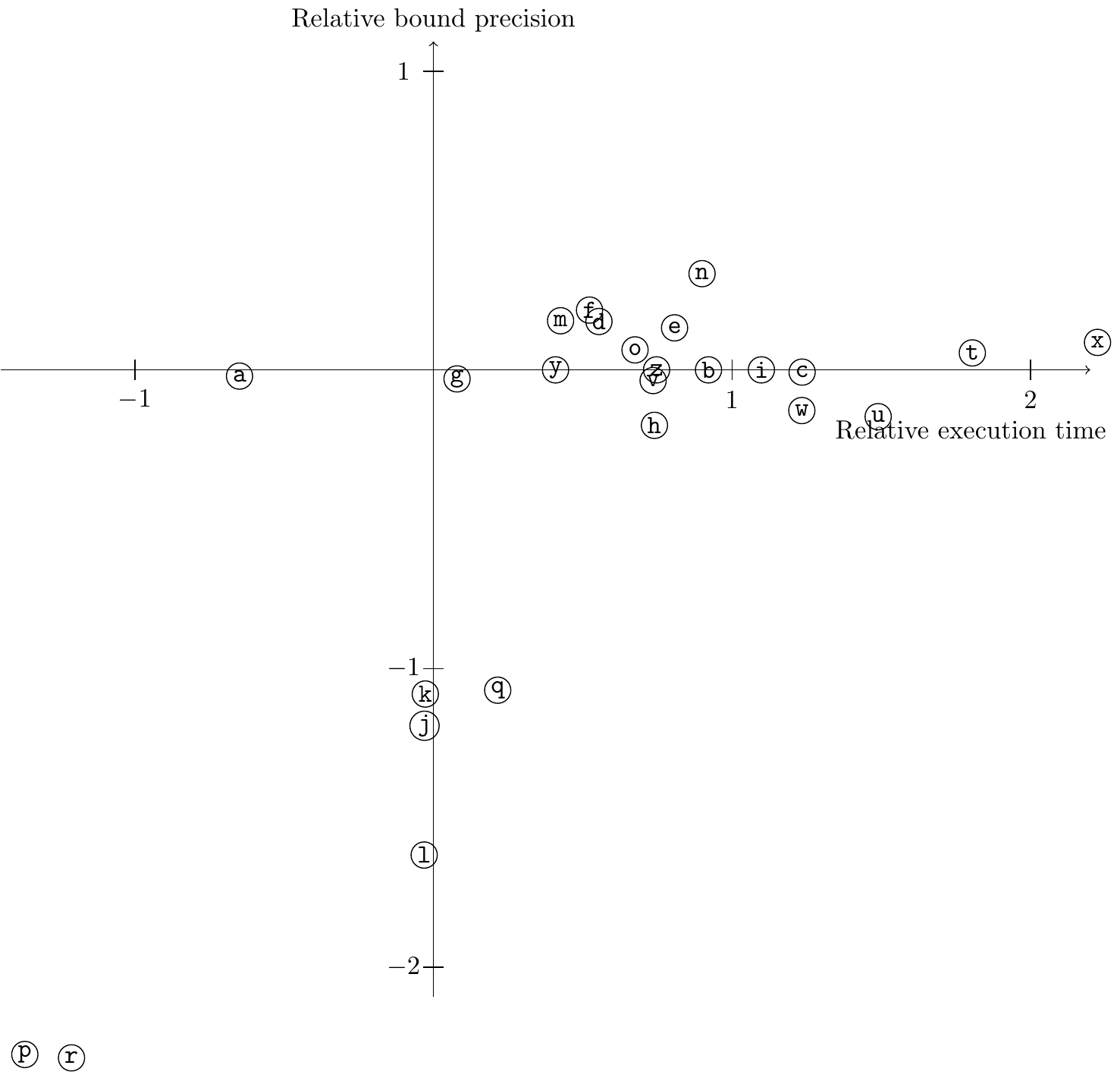}
\label{fig:timeboundrosa}
}
\subfigure[Comparison with $\fptaylor$.]{
\includegraphics[scale=0.53]{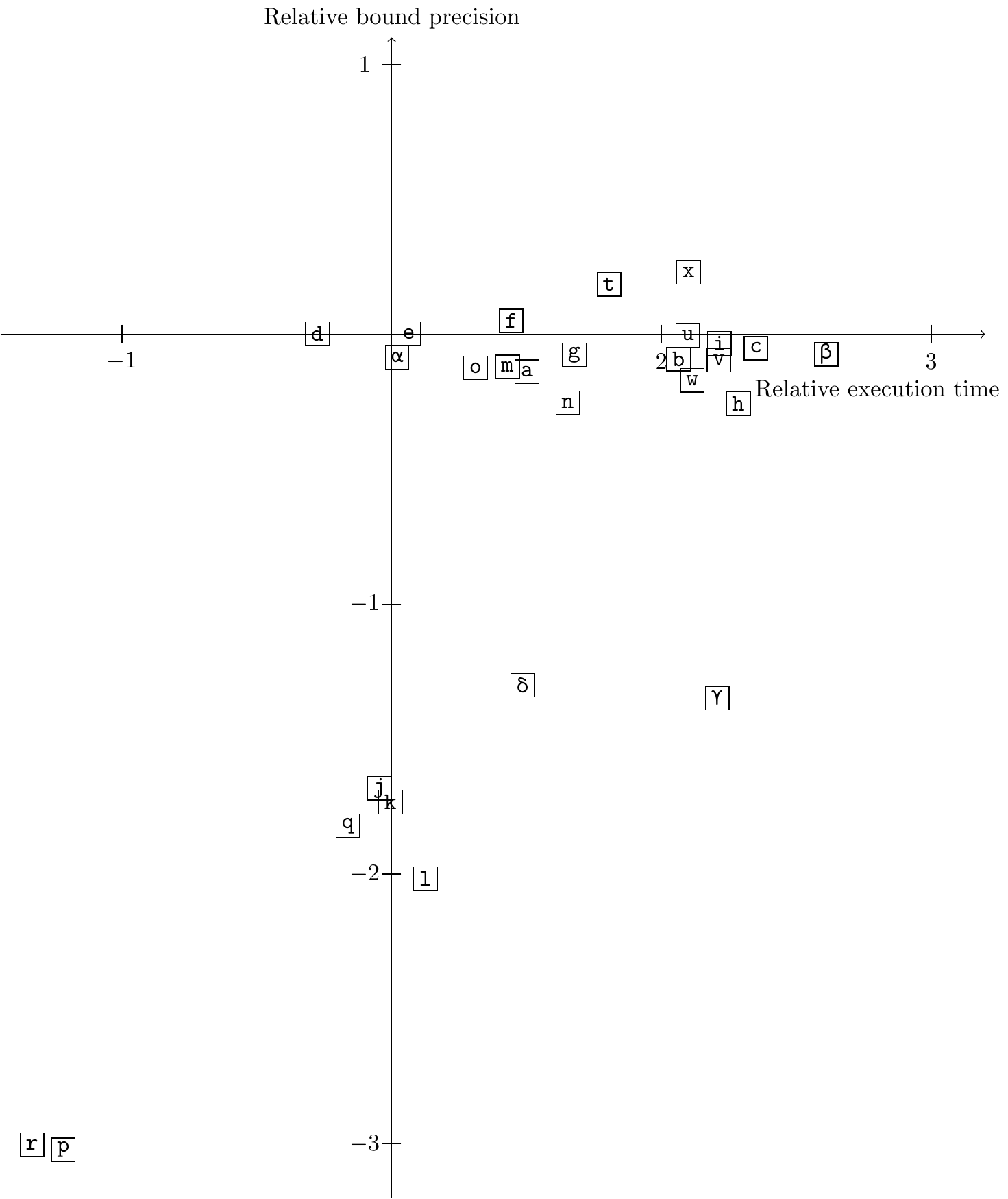}
\label{fig:timeboundfptaylor}
}
\caption{Comparisons of execution times and upper bounds of roundoff errors obtained with $\rosa$ and $\fptaylor$, relatively to $\realtofloat$.}
\label{fig:timebound}
\end{center}
\end{figure}

The axes of the coordinate system within Figure~\ref{fig:timebound} divide the plane into four quadrants: 
the nonnegative quadrant $(+,+)$ contains data points referring to programs for which $\realtofloat$ computes the tighter bounds in less time, 
the second one $(+,-)$ contains points referring to programs for which $\realtofloat$ is faster  but less accurate, 
the non-positive quadrant $(-, -)$ for which $\realtofloat$ is slower and computes coarser bounds 
and 
the last one $(-,+)$ for which $\realtofloat$ is slower but more accurate. 

On the quadrant $(+,-)$ of Figure~\ref{fig:timeboundfptaylor}, one can see that $\realtofloat$ computes bounds which are less accurate than $\fptaylor$ on semialgebraic and transcendental programs, but  does so more quickly for most of them. The quadrant $(-, -)$ indicates that $\rosa$ and $\fptaylor$ are more precise and efficient than $\realtofloat$ on the three programs \code{p-r}. The presence of 18 plots on the nonnegative quadrant $(+, +)$ of Figure~\ref{fig:timeboundrosa} and Figure~\ref{fig:timeboundfptaylor} confirms that $\realtofloat$ does not compromise efficiency at the expense of accuracy, in particular for programs implementing polynomials with nontrivial polynomial preconditions. 

%

For each program implementing polynomials, our tool has an option to provide formal guarantees for the corresponding roundoff error bound $\epsilon_{\realtofloat}$. Using the formal mechanism described in Section~\ref{sec:coqbackground}, $\realtofloat$  formally checks inside $\coq$ the SOS certificates generated by the SDP solver for interval enclosures of linear error terms $l$. 
The $\fptaylor$ (resp.~{\sc Gappa}) software has a similar option to provide formal scripts which can be checked inside the $\hol$ (resp.~$\coq$) proof assistant, for programs involving polynomial and transcendental functions.
For formal verification, our tool is limited compared with $\fptaylor$ and {\sc Gappa} as it cannot handle non-polynomial programs.

\begin{table}[!ht]
\begin{center}
\tbl{Comparisons of informal and formal execution times to certify roundoff error bounds obtained with $\realtofloat$, $\fptaylor$ and {\sc Gappa}.\label{table:cpuformal}}{
\begin{tabular}{p{2.3cm}c|ccc|ccc}
\hline
& & \multicolumn{3}{c|}{Informal execution time} & \multicolumn{3}{c}{Formal execution time} \\
\multirow{1}{*}{Benchmark} & \texttt{id} & $\realtofloat$ & $\fptaylor$  & {\sc Gappa} & $\realtofloat$ & $\fptaylor$  & {\sc Gappa}\\
\hline
\multirow{1}{*}{\texttt{rigidBody1}} & \texttt{a} &
0.58 & 1.84 & 0.10 & 0.36 & 10.2 & 2.58 \\
\multirow{1}{*}{\texttt{rigidBody2}} & \texttt{b} &
0.26 & 3.01 & 0.15 & 4.81 & 32.3 & 4.90 \\
\multirow{1}{*}{\texttt{kepler0}} & \texttt{c} & 
0.22 & 4.93 & 0.44 & 0.29 & 45.5 & 13.0 \\
\multirow{1}{*}{\texttt{kepler1}} & \texttt{d} &
17.6 & 9.33 & 0.72 & 449 & 90.5 & 28.9 \\
\multirow{1}{*}{\texttt{kepler2}} & \texttt{e} &
16.5 & 19.1 & 1.58 & 297 & 274 & 59.3 \\
\multirow{1}{*}{\texttt{sineTaylor}} & \texttt{f} &
1.05 & 2.91 & 0.16 & 7.54 & 42.1 & 13.6 \\
\multirow{1}{*}{\texttt{sineOrder3}} & \texttt{g} &
0.40 & 1.90 & 0.06 & 0.34 & 10.4 & 6.74 \\
\multirow{1}{*}{\texttt{sqroot}} & \texttt{h} &
0.14 & 2.70 & 0.19 & 0.80 & 16.8 & 12.9 \\
\multirow{1}{*}{\texttt{himmilbeau}} & \texttt{i} &
0.20 & 3.28 & 0.09 & 0.89 & 32.2 & 3.73 \\
\hline
\end{tabular}

\if{
let trel treal2float tfptaylor treal2floatformal tfptaylorformal = ( tfptaylor /.treal2float ), ( (tfptaylorformal +. tfptaylor) /. (treal2floatformal +. treal2float)) ;;

let terel i treal2float tfptaylor treal2floatformal tfptaylorformal  = 
let a, b = trel treal2float tfptaylor treal2floatformal tfptaylorformal in 
Printf.sprintf "
" 
treal2float tfptaylor treal2floatformal tfptaylorformal a b  ;;

let s1 = terel 'a' 0.58  1.84 0.36 10.2  in
let s2 = terel 'b'  0.26  3.01 4.81 32.3  in
let s3 = terel 'c'  0.22  4.93 0.29  45.5 in

let s4 = terel 'd'  17.6  9.33 449. 90.5  in
let s5 = terel 'e'  16.5  19.1 297.0 274.0 in
let s6 = terel 'f'  1.05   2.91 7.54  42.1  in
let s7 = terel 'g'  0.40  1.90 0.34 10.4   in
let s8 = terel 'h'  0.14   2.70 0.80 16.8 in
let s9 = terel 'i'  0.20   3.28 0.89 32.2 in

print_endline (s1^s2^s3 ^ s4 ^ s5 ^ s6 ^ s7 ^s8^s9);;

}\fi}
\end{center}
\end{table}

Next, we describe the formal proof results obtained while verifying the bounds for the nine polynomial programs \code{a-i}.  Table~\ref{table:cpuformal} allows us to compare the execution times of $\realtofloat$, $\fptaylor$ and {\sc Gappa} required to analyze the nine programs in both informal (i.e. without verification inside $\coq$ or $\hol$) and formal settings. 
Comparing $\realtofloat$ with $\fptaylor$, 
the latter performs better than the former to analyze formally the two programs \code{d-e} but yields coarser bounds. Our formal procedure is less computationally efficient when the degree (resp.~number of variables) of the SOS polynomials grows, as for these two programs.
The informal speedup ratio is greater than 3 for the five programs \code{a}, \code{c} and \code{g-i}. The table shows that the speedup ratio in the formal setting is higher than in the informal setting for these five programs. 
%
The explanation could be that $\coq$ is inherently faster than $\hol$ at checking computations while delegating expensive ones to a virtual machine (being part of $\coq$'s trusted base). 
Either SOS or Taylor methods could work with the two proof assistants $\hol$ and $\coq$ (with natural changes in execution time) and  it would be meaningful to compare similar methods in a given proof assistant (SOS techniques in $\hol$ or Taylor methods in $\coq$) but both are not yet fully available. 
%
The {\sc Gappa} tool performs better when analyzing the two programs \code{d-e}, but is less accurate. This is in contrast with all other programs where the formal verification with $\realtofloat$ is faster than with {\sc Gappa} while providing coarser bounds (except program  \texttt{f} where {\sc Gappa} yields pessimistic results)
Our formal verification framework is a work in progress as it only allows to check correctness of SOS certificates for polynomial programs. 
The step from formal verification of the step translating the inequality to a statement of explicit syntactic form ``$\forall \x, |\texttt{rounded}(f(\x)) - f(\x)| \leq e$'' is missing and requires more software engineering. 
Thus, the timings presented in Table~IV for $\realtofloat$ do not include formal computation of the symbolic second-order derivatives of $r$ w.r.t.~$\e$ as well as the cost of bounding them using formal interval arithmetic.

At the moment (and in contrast to our method) the $\rosa$ tool does not formally verify the output bound provided by the SMT solver, but such a feature could be embedded through an interface with the $\smtcoq$ framework~\cite{smtcoq}. This latter tool allows the proof witness generated by an SMT solver to be formally (and independently) re-checked inside $\coq$.   The $\smtcoq$ framework uses tactics based on {\em computational reflection} to enable this re-checking to be performed efficiently.

%

%
\section{CONCLUSION AND PERSPECTIVES} 
Our verification framework allows us to over-approximate roundoff errors occurring while executing nonlinear programs implemented with finite precision.
The framework relies on semidefinite optimization, ensuring certified approximations. Our approach extends to medium-size nonlinear problems, due to  automatic detection of the correlation sparsity pattern of input variables and roundoff error variables. 
Experimental results indicate that the optimization algorithm implemented in our $\realtofloat$ software package can often produce bounds of quality similar to the ones provided by the competitive solvers $\rosa$ and $\fptaylor$, while saving a significant amount of CPU time. In addition, $\realtofloat$ produces sums of squares certificates which guarantee the correctness of these upper bounds and can can be efficiently verified inside the $\coq$ proof assistant. 

This work yields several directions for further research investigation. 
First, we intend to increase the size of graspable instances by exploiting the SDP relaxations specifically tuned to the case when a program implements the sum of many rational functions~\cite{Rat15}. Symmetry patterns of certain program sub-classes could be tackled with the SDP hierarchies from~\cite{Riener2013SymmetricSDP}. We could also provide roundoff error bounds for more general programs, involving either finite or infinite conditional loops and additional comparisons with the results described in~\cite{rosaloop}. A preliminary mandatory step is to be able to generate inductive invariants with SDP relaxations. Another interesting direction would be to apply the method used in~\cite{Lasserre11} to derive sequences of lower roundoff error bounds together with SDP-based certificates.
On the formal proof side, we could benefit from floating-point/interval arithmetic libraries available inside $\coq$, first to improve the efficiency of the formal polynomial checker, currently relying on exact arithmetic, then to extend the formal verification to non-polynomial programs.
The method implemented in $\fptaylor$ happens to be efficient and precise to analyze various programs and it would be interesting to design a procedure combining $\fptaylor$ with our tool on specific subsets of input constraints. Long-term research perspectives include theoretical study of why/when SOS performs better as well as satisfactory complexity results about SOS certificates (w.r.t. size of polynomials).
%
Finally, we plan to combine this optimization framework with the procedure in~\cite{Gao15FPGA} to improve the automatic reordering of arithmetic expressions, allowing more efficient optimization of FPGA implementations.






\appendix
\section*{NONLINEAR PROGRAM BENCHMARKS}
\setcounter{section}{1}
\label{sec:appa}
{\scriptsize
\begin{lstlisting}

let box_rigidbody1 $x_1 \, x_2 \, x_3 = [(-15, 15); (-15, 15); (-15, 15)];;$
let obj_rigidbody1 $x_1 \, x_2 \, x_3 = [(
-x_1*x_2 - 2 * x_2 * x_3 - x_1 - x_3, \,0)];;$

let box_rigidbody2 $x_1 \, x_2 \, x_3 = [(-15, 15); (-15, 15); (-15, 15)];;$
let obj_rigidbody2 $x_1 \, x_2 \, x_3 = $
$[(2*x_1*x_2*x_3 + 3*x_3*x_3 - x_2*x_1*x_2*x_3 + 3*x_3*x_3 - x_2, 0)];;$

let box_kepler0 $x_1 \, x_2 \, x_3 \, x_4 \, x_5 \, x_6 = [(4, 6.36); (4, 6.36); (4, 6.36); (4, 6.36); (4, 6.36); (4, 6.36)];;$
let obj_kepler0 $x_1 \, x_2 \, x_3 \, x_4 \, x_5 \, x_6 = $
$[(x_2 * x_5 + x_3 * x_6 - x_2 * x_3 - x_5 * x_6 + x_1 * ( - x_1 + x_2 + x_3 - x_4 + x_5 + x_6), \,0)];;$

let box_kepler1 $x_1 \, x_2 \, x_3 \, x_4 = [(4, 6.36); (4, 6.36); (4, 6.36); (4, 6.36)];;$
let obj_kepler1 $x_1 \, x_2 \, x_3 \, x_4 = [( x_1 * x_4 * (- x_1 + x_2 + x_3 - x_4)$
$+ x_2 * (x_1 - x_2 + x_3 + x_4) + x_3 * (x_1 + x_2 - x_3 + x_4) $
$- x_2 * x_3 * x_4 - x_1 * x_3 - x_1 * x_2 - x_4, \,0)];; $

let box_kepler2 $x_1 \, x_2 \, x_3 \, x_4 \, x_5 \, x_6 = [(4, 6.36); (4, 6.36); (4, 6.36); (4, 6.36); (4, 6.36); (4, 6.36)];;$
let obj_kepler2 $x_1 \, x_2 \, x_3 \, x_4 \, x_5 \, x_6 = [(x_1 * x_4 * (- x_1 + x_2 + x_3$
$- x_4 + x_5 + x_6) + x_2 * x_5 * (x_1 - x_2 +x_3 +x_4- x_5 +x_6) $
$+ x_3* x_6* (x_1 + x_2 - x_3 + x_4 + x_5 - x_6) - x_2* x_3* x_4 $
$- x_1* x_3* x_5 - x_1* x_2* x_6 - x_4* x_5* x_6, \, 0)];;$

let box_sineTaylor$ \, x = [(-1.57079632679, 1.57079632679)];;$
let obj_sineTaylor$ \, x = [(x - (x*x*x)/6.0$
$+ (x*x*x*x*x)/120.0 $
$- (x*x*x*x*x*x*x)/5040.0,\, 0)];;$

let box_sineOrder3 $z = [(-2, 2)];;$
let obj_sineOrder3 $z = [(0.954929658551372 * z -0.12900613773279798*(z*z*z),\, 0)];;$

let box_sqroot $y = [(0,1)];;$
let obj_sqroot $y = [(1.0 + 0.5*y - 0.125*y*y + 0.0625*y*y*y - 0.0390625*y*y*y*y, \,0)];;$

let box_himmilbeau $x_1 \, x_2 = [(-5, 5); (-5, 5)];;$
let obj_himmilbeau $x_1 \, x_2 = [( $
$(x_1*x_1 + x_2 - 11)*(x_1*x_1 + x_2 - 11) + (x_1 + x_2*x_2 - 7)* (x_1 + x_2*x_2 - 7), \,0)];; $

let box_doppler1 $u$ $v$ $T = [(-100, 100);(20, 20e3);(-30, 50)];;$ 
let obj_doppler1 $u$ $v$ $T = [($let $t_1 = 331.4 + 0.6 * T$ in $-t_1*v/((t_1 + u)*(t_1 + u)), \,0)];;$

let box_doppler2 $u$ $v$ $T = [(-125, 125);(15, 25e3);(-40, 60)];;$ 
let obj_doppler2 $u$ $v$ $T = [($let $t_1 = 331.4 + 0.6 * T$ in $-t_1*v/((t_1 + u)*(t_1 + u)), \,0)];;$

let box_doppler3 $u$ $v$ $T = [(-30, 120);(320, 20300);(-50, 30)];;$ 
let obj_doppler3 $u$ $v$ $T = [($let $t_1 = 331.4 + 0.6 * T$ in $-t_1*v/((t_1 + u)*(t_1 + u)), \,0)];;$

let box_verhulst $x = [(0.1, 0.3)];;$
let obj_verhulst $x = [( 4 * x / (1 + (x/1.11)), \, 0)];;$

let box_carbonGas $v = [(0.1, 0.5)];;$
let obj_carbonGas $v = [($let $p = 3.5e7$ in let $a = 0.401$ in 
let $b = 42.7\text{e--}6$ in let $t = 300$ in let $n = 1000$ in
$(p + a * (n/v)**2) * (v - n * b) - 1.3806503\text{e--}23 * n * t
,\, 0)];;$

let box_predPrey$ \,x = [(0.1, 0.3)];;$
let obj_predPrey$\, x = [(4 * x * x/(1 + (x/1.11)**2), \,0)];;$

let box_turbine1 $v$ $w$ $r = [(-4.5, -0.3); (0.4, 0.9); (3.8, 7.8)];;$
let obj_turbine1 $v$ $w$ $r = [(3 + 2 / (r * r) - 0.125 * (3 - 2 * v) * (w * w * r * r) / (1 - v) - 4.5, 0)];;$

let box_turbine2 $v$ $w$ $r = [(-4.5, -0.3); (0.4, 0.9); (3.8, 7.8)];;$
let obj_turbine2 $v$ $w$ $r = [(6*v - 0.5 * v * (w*w*r*r) / (1-v) - 2.5, 0)];;$

let box_turbine3 $v$ $w$ $r = [(-4.5, -0.3); (0.4, 0.9); (3.8, 7.8)];;$
let obj_turbine3 $v$ $w$ $r = [(3 - 2/(r*r) - 0.125 * (1+2*v) * (w*w*r*r) / (1-v) - 0.5, 0)];;$

let box_jet $x_1 \, x_2 = [(-5, 5); (-20, 5)];;$
let obj_jet $x_1 \, x_2 = [($
$x_1 + ((2*x_1*((3*x_1*x_1 + 2*x_2 - x_1)/(x_1*x_1 + 1))$
$*((3*x_1*x_1 + 2*x_2 - x_1)/(x_1*x_1 + 1) - 3) $
$+ x_1*x_1*(4*((3*x_1*x_1 + 2*x_2 - x_1)/(x_1*x_1 + 1))-6))$
$* (x_1*x_1 + 1) + 3*x_1*x_1*((3*x_1*x_1 + 2*x_2 - x_1)/(x_1*x_1 + 1)) $
$+ x_1*x_1*x_1 + x_1 +    3*((3*x_1*x_1 + 2*x_2 -x_1)/(x_1*x_1 + 1))),0)];;$

let box_floudas2_6 $x_1 \, x_2 \, x_3 \, x_4 \, x_5 \, x_6 \, x_7 \, x_8 \, x_9 \, x_{10} = $ 
$[(0, 1); (0, 1); (0, 1); (0, 1); (0, 1); (0, 1); (0, 1); (0, 1); (0, 1); (0, 1)];;$
let cstr_floudas2_6  $x_1 \, x_2 \, x_3 \, x_4 \, x_5 \, x_6 \, x_7 \, x_8 \, x_9 \, x_{10} = [$
$-4  +2 * x_1 +6 * x_2 + 1 * x_3 + 0 *x_4 +3 * x_5 +3 * x_6 +2 * x_7 +6 * x_8 +2 * x9 +2 * x_{10};$
$22  -(6 * x_1 -5 * x_2 + 8 * x_3 -3 *x_4 +0 * x_5 +1 * x_6 +3 * x_7 +8 * x_8 +9 * x9 -3 * x_{10});$
$-6  -(5 * x_1 +6 * x_2 + 5 * x_3 + 3 *x_4 +8 * x_5 -8 * x_6 +9 * x_7 +2 * x_8 +0 * x9 -9 * x_{10});$
$-23 -(9 * x_1 +5 * x_2 + 0 * x_3 -9 *x_4 +1 * x_5 -8 * x_6 +3 * x_7 -9 * x_8 -9 * x9 -3 * x_{10});$
$-12 -(-8 * x_1 +7 * x_2 -4 * x_3 -5 *x_4 -9 * x_5 +1 * x_6 -7 * x_7 -1 * x_8 +3 * x9 -2 * x_{10})];;$
let obj_floudas2_6  $x_1 \, x_2 \, x_3 \, x_4 \, x_5 \, x_6 \, x_7 \, x_8 \, x_9 \, x_{10} = [($
$48 * x_1+42*x_2 + 48 * x_3 +  45  * x_4 + 44  * x_5 + 41  * x_6 + 47  * x_7$
$+ 42*x_8 + 45  * x9 + 46 *x_{10} $
$- 50 * (x_1*x_1 + x_2*x_2 + x_3*x_3 +x_4*x_4 + x_5*x_5 $ 
$+ x_6*x_6 + x_7*x_7 + x_8*x_8 + x9*x9 + x_{10}*x_{10}), 0)];;$

let box_floudas3_3 $x_1 \, x_2 \, x_3 \, x_4 \, x_5 \, x_6 = [(0, 6); (0, 6); (1, 5); (0, 6); (1, 5); (0, 10)];;$
$- (x_2 - 2)**2 - (x_3 - 1)**2 - (x_4 - 4)**2 - (x_5 - 1)**2 - (x_6 - 4)**2, \, 0)];;$
let cstr_floudas3_3 $x_1 \, x_2 \, x_3 \, x_4 \, x_5 \, x_6 = $
$[(x_3 - 3)**2 + x_4 - 4; (x_5 - 3)**2 + x_6 - 4; $
$2 - x_1 + 3 * x_2; 2 + x_1 - x_2; 6 - x_1 - x_2; x_1 + x_2 - 2];;$
let obj_floudas3_3 $x_1 \, x_2 \, x_3 \, x_4 \, x_5 \, x_6 = [( -25 * (x_1 - 2)**2 $
$- (x_2 - 2)**2 - (x_3 - 1)**2 - (x_4 - 4)**2 $
$- (x_5 - 1)**2 - (x_6 - 4)**2, \, 0)];;$

let box_floudas3_4 $x_1 \, x_2 \, x_3 = [(0, 2); (0, 2); (0, 3)];;$
let cstr_floudas3_4 $x_1 \, x_2 \, x_3 = [$
$4 - x_1 - x_2 - x_3; 6 - 3 * x_2 - x_3;$
$-0.75+2*x_1-2*x_3+4*x_1*x_1-4*x_1*x_2$
$+4*x_1*x_3+2*x_2*x_2-2*x_2*x_3+2*x_3*x_3];;$
let obj_floudas3_4 $x_1 \, x_2 \, x_3 = [(-2 * x_1 + x_2 - x_3, 0)];;$

let box_floudas4_6$\,x_1 \, x_2 = [(0, 3); (0, 4)];;$
let cstr_floudas4_6$\,x_1 \, x_2 = [$
$2 * x_1**4 - 8 * x_1**3 + 8 * x_1*x_1 - x_2; $
$4 * x_1**4 - 32 * x_1**3 + 88 * x_1*x_1 - 96 * x_1 + 36 - x_2];;$
let obj_floudas4_6$\,x_1 \, x_2 = [(-x_1 - x_2, \, 0)];;$

let box_floudas4_7$\,x_1 \, x_2 = [(0, 2); (0, 3)];;$
let cstr_floudas4_7$\,x_1 \, x_2 = [-2 * x_1**4 + 2 - x_2];;$
let obj_floudas4_7$\,x_1 \, x_2 = [(-12 * x_1 - 7 * x_2 + x_2*x_2, \, 0)];;$

let box_cav10 $x = [(0, 10)];;$
let obj_cav10 $x = [($ if $(x*x - x > 0)$ then $x*0.1$ else $x*x+2, \, 0)];;$

let box_perin $x \, y = [(1,7); (-2, 7)];;$
let cstr_perin $x \, y = [x-1; y+2; x-y; 5-y-x];;$
let obj_perin $x \, y = [($ if $(x*x + y*y \leq 4)$ then $y * x$ else $0, \, 0)];; $

let box_logexp $x = [(-8,8)];;$
let obj_logexp $x = [(\log(1 + \exp(x)), \, 0)];;$

let box_sphere $x \, r \, y \, z = [(-10, 10); (0, 10); (-1.570796, 1.570796); (-3.14159265, 3.14159265)];;$
let obj_sphere $x \, r \, y \, z = [(x + r * \sin (y) * \cos(z),\,0)];;$

let box_hartman3 $x_1 \, x_2 \, x_3 = [(0, 1); (0, 1);(0, 1)];;$
let obj_hartman3 $x_1 \, x_2 \, x_3 = [($
let $e1 = 3.0 * (x_1 - 0.3689) **2 + 10.0 * (x_2 - 0.117) **2 + 30.0 * (x_3 - 0.2673) **2$ in
let $e2 = 0.1 * (x_1 - 0.4699) **2 + 10.0 * (x_2 - 0.4387) **2 + 35.0 * (x_3 - 0.747) **2$ in
let $e3 = 3.0 * (x_1 - 0.1091) **2 + 10.0 * (x_2 - 0.8732) **2 + 30.0 * (x_3 - 0.5547) **2$ in
let $e4 = 0.1 * (x_1 - 0.03815) **2 + 10.0 * (x_2 - 0.5743) **2 + 35.0 * (x_3 - 0.8828) **2$ in
$- (1.0 * \exp(-e1) + 1.2 * \exp(-e2) + 3.0 * \exp(-e3) + 3.2 * \exp(-e4)), \, 0)];;$

let box_hartman6 $x_1 \, x_2 \, x_3 \, x_4 \, x_5 \, x_6 = [(0, 1); (0, 1);(0, 1);(0, 1);(0, 1);(0, 1)];;$
let obj_hartman6 $x_1 \, x_2 \, x_3 \, x_4 \, x_5 \, x_6 = [($
let $e1 = 10.0 * (x_1 - 0.1312)**2 + 3.0 * (x_2 - 0.1696)**2 + 17. * (x_3 - 0.5569)**2 + 3.5 * (x_4 - 0.0124)**2 $
$+ 1.7 * (x_5 - 0.8283)**2 + 8.0 * (x_6 - 0.5886)**2$ in
let $e2 = 0.05 * (x_1 - 0.2329)**2 + 10 * (x_2 - 0.4135)**2 + 17.0 * (x_3 - 0.8307)**2 + 0.1 * (x_4 - 0.3736)**2$
$+ 8.0 * (x_5 - 0.1004)**2 + 14.0 * (x_6 - 0.9991)**2$ in
let $e3 = 3.0 * (x_1 - 0.2348)**2 + 3.5 * (x_2 - 0.1451)**2 + 1.7 * (x_3 - 0.3522)**2 + 10.0 * (x_4 - 0.2883)**2$
$+ 17.0 * (x_5 - 0.3047)**2 + 8.0 * (x_6 - 0.665)**2$ in
let $e4 = 17.0 * (x_1 - 0.4047)**2 + 8 * (x_2 - 0.8828)**2 + 0.05 * (x_3 - 0.8732)**2 + 10.0 * (x_4 - 0.5743)**2$
$+ 0.1 * (x_5 - 0.1091)**2 + 14.0 * (x_6 - 0.0381)**2$ in
$- (1.0 * \exp(-e1) + 1.2 * \exp(-e2) + 3.0 * \exp(-e3) + 3.2 * \exp(-e4)), \, 0)];;$




\end{lstlisting}
}

\if{
{\scriptsize
\begin{lstlisting}
let box_doppler1 u v T = [$(-100, 100); (20, 2e4);(-30, 50)$];; 
let obj_doppler1 u v T = [(
let $t_1 = 331.4 + 0.6 * T$ in 
$-t_1*v/((t_1 + u)*(t_1 + u))$, $3\text{e--}13$)];;
\end{lstlisting}
}
{\scriptsize
\begin{lstlisting}
procedure doppler2(u : real, v : real, T : real) returns (r : real) {
 assume (-125.0 <= u && u <= 125.0 && 15.0 <= v && v <= 25000.0 && -40.0 <= T && T <= 60.0);
  var t1 := 331.4 + 0.6 * T;
  r := -t1*v/((t1 + u)*(t1 + u));
}
\end{lstlisting}
}
{\scriptsize
\begin{lstlisting}
procedure doppler3(u : real, v : real, T : real) returns (r : real) {
 assume (-30.0 <= u && u <= 120.0 && 320.0 <= v && v <= 20300.0 && -50.0 <= T && T <= 30.0);
  var t1 := 331.4 + 0.6 * T;
  r := -t1*v/((t1 + u)*(t1 + u));
}
\end{lstlisting}
}
{\scriptsize
\begin{lstlisting}
procedure rigidBody1(x_1 : real, x_2 : real, x_3 : real) returns (r : real) {
 assume (-15.0 <= x_1 && x_1 <= 15.0 && -15.0 <= x_2 && x_2 <= 15.0 && -15.0 <= x_3 && x_3 <= 15.0);
  r := -x_1*x_2 - 2.0 * x_2 * x_3 - x_1 - x_3;
}
\end{lstlisting}
}
{\scriptsize
\begin{lstlisting}
procedure rigidBody2(x_1 : real, x_2 : real, x_3 : real) returns (r : real) {
 assume (-15.0 <= x_1 && x_1 <= 15.0 && -15.0 <= x_2 && x_2 <= 15.0 && -15.0 <= x_3 && x_3 <= 15.0);
  r := 2.0*x_1*x_2*x_3 + 3.0*x_3*x_3 - x_2*x_1*x_2*x_3 + 3.0*x_3*x_3 - x_2;
}
\end{lstlisting}
}
{\scriptsize
\begin{lstlisting}
procedure sineTaylor(x : real) returns (r : real) {
 assume (-1.57079632679  <= x && x <= 1.57079632679);
    r := x - (x*x*x)/6.0 + (x*x*x*x*x)/120.0 - (x*x*x*x*x*x*x)/5040.0 ;
}
\end{lstlisting}
}
{\scriptsize
\begin{lstlisting}
procedure sineOrder3(x : real) returns (r : real) {
 assume (-2  <= x && x <= 2);
  r := 0.954929658551372 * x - 0.12900613773279798*x*x*x ;
}
\end{lstlisting}
}
{\scriptsize
\begin{lstlisting}
procedure sqroot(x : real) returns (r : real) {
 assume (0  <= x && x <= 1);
  r := 1.0 + 0.5*x - 0.125*x*x + 0.0625*x*x*x - 0.0390625*x*x*x*x;
}
\end{lstlisting}
}
}\fi

\section*{Acknowledgment}
The authors would like to specially acknowledge the precious help of Alexey Solovyev, his excellent feedback and suggestions. 
They also thank the three referees for helpful comments to improve this paper.

\bibliographystyle{ACM-Reference-Format-Journals}


\end{document}